\documentclass[reqno, 12pt]{amsart}
\usepackage[utf8]{inputenc}
\usepackage[english]{babel}
\usepackage{fullpage}
\usepackage{amssymb,amsmath,amscd,mathrsfs,graphicx,amsthm,stmaryrd}
\usepackage{bm}
\usepackage{rotating, enumerate}
\usepackage{tikz-cd, tikz, verbatim}	

\newtheorem{thm}{Theorem}[section]
\newtheorem{prop}[thm]{Proposition}
\newtheorem{lemma}[thm]{Lemma}

\theoremstyle{definition}

\newtheorem*{ack}{\textbf{Ackowledgments}}
\newtheorem*{conv}{\textbf{Conventions}}
\newtheorem{prob}[thm]{Problem}
\theoremstyle{remark}
\newtheorem{rmk}[thm]{Remark}
\newtheorem*{rmk*}{Remark}

\newcommand{\Img}{\mathrm{Im}}
\newcommand{\Aff}{\mathbb{A}}
\newcommand{\PP}{\mathbb{P}}
\newcommand{\CC}{\mathbb{C}}
\newcommand{\QQ}{\mathbb{Q}}
\newcommand{\ZZ}{\mathbb{Z}}
\newcommand{\RR}{\mathbb{R}}
\newcommand{\FF}{\mathbb{F}}
\newcommand{\cO}{\mathcal{O}}
\newcommand{\cM}{\mathcal{M}}
\newcommand{\cF}{\mathcal{F}}
\newcommand{\cP}{\mathcal{P}}
\newcommand{\cI}{\mathcal{I}}
\newcommand{\NS}{\mathrm{NS}}
\newcommand{\Pic}{\mathrm{Pic}}
\newcommand{\rk}{\mathrm{rk}}
\newcommand{\Or}{\mathrm{O}}
\newcommand{\divi}{\mathrm{div}}
\newcommand{\id}{\mathrm{id}}

\title{The Kodaira dimension of some moduli spaces of elliptic K3 surfaces}
\author{Mauro Fortuna}
\address{Institut für Algebraische Geometrie, Leibniz Universität Hannover, Welfengarten 1, 30167 Hannover, Germany.}
\email{fortuna@math.uni-hannover.de}

\author{Giacomo Mezzedimi}
\address{Institut für Algebraische Geometrie, Leibniz Universität Hannover, Welfengarten 1, 30167 Hannover, Germany.}
\email{mezzedimi@math.uni-hannover.de}

\makeatletter
\@namedef{subjclassname@2020}{%
  \textup{2020} Mathematics Subject Classification}
\makeatother


\begin{document}

\maketitle

\begin{abstract}
We study the moduli spaces of elliptic K3 surfaces of Picard number at least 3, i.e.\ $U\oplus \langle -2k \rangle$-polarized K3 surfaces. Such moduli spaces are proved to be of general type for $k\geq 220$. The proof relies on the low-weight cusp form trick developed by Gritsenko, Hulek and Sankaran. Furthermore, explicit geometric constructions of some elliptic K3 surfaces lead to the unirationality of these moduli spaces for $k < 11$ and for 19 other isolated values up to $k=64$.
\end{abstract}

\section*{Introduction}

Moduli spaces of complex K3 surfaces are a fundamental topic of interest in algebraic geometry. One of the first geometric properties one wants to understand is their Kodaira dimension. Towards this direction, the seminal work \cite{GHS07} of Gritsenko, Hulek and Sankaran proved that the moduli space $\cF_{2d}$ of polarized K3 surfaces of degree $2d$ is of general type for $d>61$ and for other smaller values of $d$. It is then natural to address the general question about the Kodaira dimension of moduli spaces of lattice polarized K3 surfaces. We are interested in studying a particular class of such surfaces, namely elliptic K3 surfaces of Picard number at least 3.

A K3 surface $X$ is called elliptic if it admits a fibration $X\rightarrow \PP^1$ in curves of genus one together with a section. The classes of the fiber and the zero section in the N\'eron-Severi group generate a lattice isomorphic to the hyperbolic plane $U$, and they span the whole N\'eron-Severi group if the elliptic K3 surface is very general. The geometry of elliptic surfaces can be studied via their realization as Weierstrass fibrations. By using this description,  Miranda \cite{Mir81} constructed the moduli space of elliptic K3 surfaces and showed its unirationality as a by-product. Later, Lejarraga \cite{Lej93} proved that this space is actually rational. We want to study the divisors of the moduli space of elliptic K3 surfaces which parametrize the surfaces whose Néron-Severi groups contain primitively $U\oplus \langle -2k \rangle$. These are the moduli spaces $\cM_{2k}$ of $U\oplus \langle -2k\rangle$-polarized K3 surfaces. Geometrically we are considering elliptic K3 surfaces admitting an extra class in the N\'eron-Severi group: if $k=1$, it comes from a reducible fiber of the elliptic fibration, while if $k\ge 2$ it is represented by an extra section, intersecting the zero section in $k-2$ points with multiplicity (cf. Remark \ref{rmk:weiestrass}).

In the present article, we aim at computing the Kodaira dimension of the moduli spaces $\cM_{2k}$.

\begin{thm}\label{thm:main}
The moduli space $\cM_{2k}$ is of general type for $k\geq 220$, or
$$ k \ge 208, \  k\ne 211,219, \text{ or }  k\in \{170,185,186,188,190,194,200,202,204,206\}. $$

Moreover, the Kodaira dimension of $\cM_{2k}$ is non-negative for $k\geq 176$, or
$$ k\ge 164,  \ k \ne 169,171,175 \text{ or } k\in \{140,146,150,152,154,155,158,160,162\}.$$
\end{thm}

The Torelli theorem for K3 surfaces (see \cite{PSS72}) allows the moduli spaces $\cM_{2k}$ to be realized as quotients of bounded hermitian symmetric domains $\Omega_{L_{2k}}$ of type IV and dimension 17 by the stable orthogonal groups $\widetilde{\Or}^+(L_{2k})$, where the lattice $L_{2k}$ is the orthogonal complement of $U\oplus \langle -2k \rangle$ in the K3 lattice $\Lambda_{K3}:=3U\oplus 2E_8(-1)$. Via this description, one can apply the low-weight cusp form trick (Theorem \ref{thm:trick}) developed in \cite{GHS07}. This tool provides a sufficient condition for an orthogonal modular variety to be of general type. Namely, one has to find a non-zero cusp form on $\Omega_{L_{2k}}^{\bullet}$ of weight strictly less than the dimension of $\Omega_{L_{2k}}$ vanishing along the ramification divisor of the projection $\Omega_{L_{2k}} \rightarrow  \widetilde{\Or}^+(L_{2k}) \backslash \Omega_{L_{2k}}$. In our case, to construct a suitable cusp form, we use the quasi-pullback method (Theorem \ref{thm:quasi}) to pull back the Borcherds form $\Phi_{12}$ along the inclusion $\Omega_{L_{2k}}^\bullet \hookrightarrow \Omega_{\mathrm{II}_{2,26}}^\bullet$ induced by a lattice embedding $L_{2k}\hookrightarrow \mathrm{II}_{2,26}$. Here, the lattice $\mathrm{II}_{2,26}$ denotes the unique (up to isometry) even unimodular lattice of signature $(2, 26)$. The lattice embedding $L_{2k}\hookrightarrow \mathrm{II}_{2,26}$ determines the number $N(L_{2k})$ of effective roots in $L_{2k}^\perp$. If $N(L_{2k})$ is positive, the embedding determines the weight $12+N(L_{2k})$ of the cusp form. Therefore the whole proof of Theorem \ref{thm:main} boils down to finding the values of $k$ for which there exists a suitable primitive embedding $L_{2k}\hookrightarrow \mathrm{II}_{2,26}$, whose orthogonal complement contains at least 2 and at most 8 roots (cf. Problem \ref{prob}). 

In the second part of the article we give a geometric construction of all $U\oplus \langle -2k \rangle$-polarized K3 surfaces as double covers of the Hirzebruch surface $\FF_4$ branched over a suitable  smooth curve admitting a rational curve intersecting the branch locus with even multiplicities. We then recall that, for $k\geq 4$ even, all $U\oplus \langle -2k \rangle$-polarized K3 surfaces admit a structure as hyperelliptic quartic K3 surfaces, i.e.\ double covers of $\PP^1\times \PP^1$ branched over a curve of bidegree $(4,4)$. Finally, we recall the realization of elliptic K3 surfaces as Weierstrass fibrations. These geometric constructions lead to the following:

\begin{thm}\label{thm:unirational}
The moduli space $\cM_{2k}$ is unirational for $k < 11$ and for $k\in \{13,$ $16,$ $17,$ $19,$ $21,$ $25,$ $26,$ $29,$ $31,$ $34,$ $36,$ $37,$ $39,$ $41,$ $43,$ $49,$ $59,$ $61,$ $64\}$.
\end{thm}

We notice that the unirationality of $\cM_{56}$ is showed in \cite{BH17} using relative canonical resolutions. 

The article is organized as follows. In Section \ref{sec:moduli} we review the general construction for the moduli spaces of lattice polarized K3 surfaces as orthogonal modular varieties. We give a description of the moduli spaces $\cM_{2k}$, which are the main object of study in this article. In Section \ref{sec:trick} we describe the method used in proving Theorem \ref{thm:main}, namely the low-weight cusp form trick (Theorem \ref{thm:trick}). The desired form is cooked up as a quasi-pullback of the Borcherds form $\Phi_{12}$ (Theorem \ref{thm:quasi}). Section \ref{sec:reflections} is devoted to the proof of Proposition \ref{prop:ram}. Indeed, we study some special reflections in the stable orthogonal group $\widetilde{\Or}^+(L_{2k})$. This is then used to impose the vanishing of the quasi-pullback $\Phi|_{L_{2k}}$ of the Borcherds form along the ramification divisor of the quotient map $\Omega_{L_{2k}} \rightarrow \cM_{2k}$. In Section \ref{sec:lattice} we tackle Problem \ref{prob} of finding primitive embeddings $L_{2k}\hookrightarrow \mathrm{II}_{2,26}$ with at least 2 and at most 8 orthogonal roots. First, we prove that for any $k\geq 4900$ such an embedding exists. Then, we perform an exhaustive computer analysis to find explicit embeddings for the remaining values of $k$. It relies on the geometry of K3 surfaces with N\'eron-Severi group isometric to $U\oplus E_8(-1)$. In Section \ref{sec:constructions} we review the classical constructions of elliptic K3 surfaces as double covers of $\PP^1 \times \PP^1$ and $\FF_4$, and Weierstrass fibrations. Finally, in Section \ref{sec:unirational} explicit geometric constructions, such as the ones presented in Section \ref{sec:constructions}, are used to prove Theorem \ref{thm:unirational}.

\begin{conv}
Throughout the article we will always work over $\CC$. We have used the software \texttt{Magma} to implement our algorithms.
\end{conv}

\begin{ack}
We would like to thank our PhD advisors Klaus Hulek and Matthias Schütt for many useful discussions and for reading an early draft of this manuscript. We warmly thank Michael Hoff and Xavier Roulleau for insightful comments and discussions. We would like to thank the anonymous referee for the helpful comments. The first author acknowledges partial support from the DFG Grant Hu 337/7-1.
\end{ack}
\section{Moduli spaces of lattice polarized K3 surfaces}\label{sec:moduli}

In this section we review the construction of the moduli spaces of lattice polarized K3 surfaces. An excellent reference to this subject is \cite{Dol96}. 

First we recall some basic notions of lattice theory. Let $L$ be an integral lattice of signature $(2, n)$. Let $\Omega_L$ be one of the two connected components of
$$ \{ [w]\in \PP(L\otimes \CC) \ |\ (w, w)=0, \ (w, \Bar{w})>0 \}.$$
It is a hermitian symmetric domain of type IV and dimension $n$. We denote by $\Or^+(L)$ the index two subgroup of the orthogonal group $\Or(L)$ preserving $\Omega_L$. If $\Gamma < \Or^+(L)$ is of finite index we denote by $\cF_L(\Gamma)$ the quotient $\Gamma \backslash \Omega_L$. By a result of Baily and Borel \cite{BB66}, $\cF_L(\Gamma)$ is a quasi-projective variety of dimension $n$. 

For every non-degenerate integral lattice $L$ we denote by $L^\vee :=\mathrm{Hom}(L, \ZZ)$ its dual lattice. If $L$ is even, the finite group $A_L:=L^\vee/L$ is endowed with a quadratic form $q_L$ with values in $\QQ/2\ZZ$, induced by the quadratic form on $L$. We define:
$$\widetilde{\Or}(L):=\ker (\Or(L)\rightarrow \Or(A_L))$$
and 
$$ \widetilde{\Or}^+(L):=\widetilde{\Or}(L) \cap \Or^+(L).$$

A compact smooth complex surface $X$ is a \textit{K3 surface} if $X$ is simply connected and $H^0(X, \Omega_X^2)$ is spanned by a non-degenerate holomorphic 2-form $\omega_X$. The cohomology group $H^2(X, \ZZ)$ is naturally endowed with a unimodular intersection
pairing, making it isomorphic to the K3 lattice
$$\Lambda_{K3}:=3U\oplus 2E_8(-1),$$
where $U$ is the hyperbolic plane and $E_8(-1)$ is the unique (up to isometry) even unimodular negative definite
lattice of rank $8$. In particular the signature of $H^2(X, \ZZ)$ is $(3, 19)$. 

Fix an integral even lattice $M$ of signature $(1, t)$ with $t\geq 0$. An \textit{$M$-polarized K3 surface} is a pair $(X, j)$ where $X$ is a K3 surface and $j:M \hookrightarrow \NS(X)$ is a primitive embedding. Let 
$$N:=j(M)^{\perp}_{\Lambda_{K3}}$$
be the orthogonal complement of $M$ in $\Lambda_{K3}$. It is an integral even lattice of signature $(2, 19-t)$.

By the Torelli theorem \cite{PSS72} (see also \cite[Corollary 3.2]{Dol96}), the moduli spaces of $M$-polarized K3 surfaces can be identified with the quotient of a classical hermitian symmetric domain of type IV and dimension $19-t$ by an arithmetic group. More precisely, the 2-form $\omega_X$ of a $M$-polarized K3 surface $X$ determines a point in the symmetric space $\Omega_N$ (the \textit{period domain}), unique up to the action of the group \cite[Proposition 3.3]{Dol96}
$$\widetilde{\Or}^+(N)=\{ g\in \Or^+(\Lambda_{K3}) \ | \ g|_M=\id \}.$$

\begin{thm} \cite[\S 3]{Dol96}
The variety $\cF_N(\widetilde{\Or}^+(N))$ is isomorphic to the coarse moduli space of $M$-polarized K3 surfaces.
\end{thm}

In the following, we will study the moduli spaces of $M$-polarized K3 surfaces with $M=U\oplus \langle -2k\rangle$, i.e.\ elliptic K3 surfaces of Picard rank at least 3. Since the embedding $U\oplus \langle -2k \rangle \hookrightarrow \Lambda_{K3}$ is unique up to isometry by \cite[Theorem 1.14.4]{Nik79}, we get the isomorphism
$$ L_{2k}:=U\oplus 2E_8(-1) \oplus \langle 2k \rangle \cong (U\oplus \langle -2k \rangle)^\perp_{\Lambda_{K3}}. $$
As we discussed above, the quotient variety
\begin{equation} \label{eq:M2k}
    \cM_{2k}:= \cF_{L_{2k}}(\widetilde{\Or}^+(L_{2k}))
\end{equation}
is the moduli space of $U\oplus \langle -2k\rangle$-polarized K3 surfaces. Notice that all these surfaces are elliptic, since their Picard lattices contain a copy of the hyperbolic plane $U$.

\section{Low-weight cusp form trick}\label{sec:trick}

The computation of the Kodaira dimension of modular orthogonal varieties relies on the \textit{low-weight cusp form trick} developed by Gritsenko, Hulek and Sankaran \cite{GHS07}. In order to describe it, we need a little theory of modular forms on orthogonal groups. 

Let $L$ be an integral even lattice of signature $(2, n)$. A \textit{modular form} of weight $k$ and character $\chi:\Gamma \rightarrow \CC^*$ for a finite index subgroup $\Gamma < \Or^+(L)$ is a holomorphic function $F: \Omega_L^{\bullet}\rightarrow \CC$ on the affine cone $\Omega_L^{\bullet}$ over $\Omega_L$ such that 
$$ F(tZ)=t^{-k}F(Z) \; \forall t \in \CC^*, \quad \mathrm{and} \quad F(gZ)=\chi(g)F(Z) \; \forall g\in \Gamma.$$
A modular form is a \textit{cusp form} if it vanishes at every cusp. We denote the vector spaces of modular forms and cusp forms of weight $k$ and character $\chi$ for $\Gamma$ by $M_k(\Gamma, \chi)$ and $S_k(\Gamma, \chi)$ respectively.

\begin{thm}\label{thm:trick} \cite[Theorem 1.1]{GHS07}
Let $L$ be an integral lattice of signature $(2, n)$ with $n\geq 9$, and let $\Gamma < \Or^+(L)$ be a subgroup of finite index. The modular variety $ \cF_L(\Gamma)$ is of general type if there exists a nonzero cusp form $F\in S_k(\Gamma, \chi)$ of weight $k<n$ and character $\chi$ that vanishes along the ramification divisor of the projection $\pi:\Omega_{L} \rightarrow \cF_L(\Gamma)$ and vanishes with order at least 1 at infinity. 

If $S_n(\Gamma, \det)\neq 0$ then the Kodaira dimension of $ \cF_L(\Gamma)$ is non-negative.
\end{thm}

\begin{rmk*}
By \cite[Theorem 1.1]{Ma18} there are only finitely many integral lattices $L$ of signature $(2, n)$ with $n\geq 21$ or $n=17$ such that $ \cF_L(\Gamma)$ is not of general type. Therefore, our moduli
spaces $\cM_{2k}$ are known to be of general type for $k$ large enough.
\end{rmk*}

\begin{rmk*}
In the recent paper \cite{Ma21} the author shows the necessity for an additional hypothesis in Theorem \ref{thm:trick} concerning the so-called irregular cusps (cf. \cite[Theorem 1.2]{Ma21}). However, this does not affect our case as explained in \cite[Example 4.10]{Ma21}.
\end{rmk*}

\subsection{Ramification divisor}
First, we need to describe the ramification divisor of the orthogonal projection, which turns out to be the union of rational quadratic divisors associated to reflective vectors.

For any $v \in L\otimes \QQ$ such that $v^2<0$ we define the \textit{rational quadratic divisor}
$$ \Omega_v(L):=\{[Z]\in \Omega_L \ | \ (Z, v)=0 \}\cong \Omega_{v^{\perp}_L}$$
where $v^{\perp}_L$ is an even integral lattice of signature $(2,n-1)$.

The reflection with respect to the hyperplane defined by a non-isotropic vector $r\in L$ is given by
$$\sigma_r:l\longmapsto l-2\frac{(l,r)}{r^2}r.$$
If $r$ is primitive and $\sigma_r \in \Or(L)$, then we say that $r$ is a \textit{reflective vector}. We notice that $r$ is always reflective if $r^2=\pm 2$, and we call it \textit{root} in this case.

If $v\in L^\vee$ and $v^2<0$, the divisor $\Omega_v(L)$ is called a \textit{reflective divisor} if $\sigma_v \in \Or(L)$.
\begin{thm} \cite[Corollary 2.13]{GHS07} \label{thm:Rdiv}
For $n\geq 6$, the ramification divisor of the projection
$ \pi_{\Gamma}: \Omega_L \rightarrow \cF_L(\Gamma) $
is the union of the reflective divisors with respect to $\Gamma < \Or^+(L)$:
$$ \mathrm{Rdiv}(\pi_{\Gamma})=\bigcup_{\substack{\ZZ r\subset L\\ \sigma_r\in \Gamma \cup -\Gamma}} \Omega_r(L)$$
\end{thm}

\subsection{Quasi pullback}
To apply Theorem \ref{thm:trick}, we need a supply of modular forms for $\Gamma$. These are provided by quasi-pullbacks of modular forms with respect to some higher rank orthogonal group. In our case, let $\mathrm{II}_{2,26}$ denote the unique (up to isometry) even unimodular lattice of signature $(2, 26)$:
$$\mathrm{II}_{2,26}:=2U\oplus 3E_8(-1).$$
Borcherds proved \cite{B95} that $M_{12}(\Or^+(\mathrm{II}_{2,26}), \det)$ is a 1-dimensional complex vector space spanned by a modular form $\Phi_{12}$, called the \textit{Borcherds form}. The zeroes of $\Phi_{12}$ lie on rational quadratic divisors defined by $(-2)$-vectors in $\mathrm{II}_{2,26}$, i.e.\ $\Phi_{12}(Z)=0$ if and only if there exists $r\in \mathrm{II}_{2,26}$ with $r^2=-2$ such that $(Z, r)=0$. Moreover the multiplicity of the rational quadratic divisor of zeroes of $\Phi_{12}$ is one.

Given a primitive embedding of lattices $L \hookrightarrow \mathrm{II}_{2,26}$, with $L$ of signature $(2, n)$, we define
$$R_{-2}(L^\perp):= \{r\in \mathrm{II}_{2,26} \ | \ r^2=-2, \ (r, L)=0 \}.$$
To construct a modular form for some subgroup of $\Or^+(L)$, one might try to pull back $\Phi_{12}$ along the closed immersion $\Omega_L^{\bullet}\hookrightarrow \Omega_{\mathrm{II}_{2,26}}^{\bullet}$. However, for any $r\in R_{-2} (L^\perp) $ one has $ \Omega_L^{\bullet}\subset \Omega_{r^{\perp}}^{\bullet}$ and hence $\Phi_{12}$ vanishes identically on $\Omega_L^{\bullet}$. The method of the \textit{quasi-pullback}, developed by Gritsenko, Hulek, and Sankaran \cite{GHS07}, deals with this issue by dividing out by appropriate linear factors:

\begin{thm}\label{thm:quasi} \cite[Theorem 8.3]{GHS15}
Let $L\hookrightarrow \mathrm{II}_{2,26}$ be a primitive non-degenerate sublattice of signature $(2, n)$, $n\geq 3$, and let $\Omega_L \hookrightarrow \Omega_{\mathrm{II}_{2,26}}$ be the corresponding embedding of the homogeneous domains. The set of $(-2)$-roots $R_{-2}(L^\perp)$ in the orthogonal complement of $L$ is finite. We put $N(L):= |R_{-2}(L^\perp)|/ 2 $. Then the function 
\begin{equation}\label{eq:quasi}
    \Phi|_{L}(Z):=\frac{\Phi_{12}(Z)}{\Pi_{r\in R_{-2}(L^\perp)/\pm 1}(Z, r)} \bigg \vert_{\Omega_L}  \in M_{12+N(L)}(\widetilde{\Or}^+(L), \det)
\end{equation}
is non-zero, where in the product over $r$ we fix a system of representatives in $R_{-2}(L^\perp)/\pm 1$. The modular form $\Phi|_{L}$ vanishes only on rational quadratic divisors of type $\Omega_v(L)$ where $v\in L^\vee$ is the orthogonal projection with respect to $L^\perp$ of a $(-2)$-root $r\in \mathrm{II}_{2,26}$ on $L^\vee$. 

Moreover, if $N(L)>0$, then $\Phi|_{L}$ is a cusp form.
\end{thm}

We want to apply the low-weight cusp form trick and Theorem \ref{thm:quasi} to the orthogonal variety isomorphic to the moduli space of $U\oplus \langle -2k \rangle$-polarized K3 surfaces.

In our situation, we need to find a suitable primitive embedding of $L_{2k} \hookrightarrow \mathrm{II}_{2,26}$, such that the quasi-pullback $\Phi|_{L_{2k}}$ is a cusp form of weight (strictly) less than 17 which vanishes along the ramification divisor of the projection 
$$ \pi: \Omega_{L_{2k}} \rightarrow \cM_{2k}=\widetilde{\Or}^+(L_{2k}) \backslash \Omega_{L_{2k}}. $$

\begin{rmk}
By \cite[Theorem 1.7]{GHS09} the abelianization of $\widetilde{\Or}^+(L_{2k})$ is isomorphic to $\ZZ/2\ZZ$. This is because $L_{2k}$ is isomorphic to $2U\oplus E_8(-1)\oplus \langle -2k\rangle^\bot_{E_8(-1)}$, since the embedding $U\oplus \langle -2k\rangle \hookrightarrow \Lambda_{K3}$ is unique up to isometry (cf. \cite[Theorem 1.14.4]{Nik79}). As a consequence, the Albanese varieties of the moduli spaces $\cM_{2k}$ are all trivial (cf. \cite[Theorem 2.5]{Kon88}). Moreover, \cite[Corollary 1.8]{GHS09} implies that the unique non-trivial character of $\widetilde{\Or}^+(L_{2k})$ is $\det$.
\end{rmk}

\section{Special reflections}\label{sec:reflections}

Let $L_{2k}\hookrightarrow \mathrm{II}_{2,26}$ be a primitive embedding. Since the embedding $U\oplus 2 E_8(-1)\hookrightarrow \mathrm{II}_{2,26}$ is unique up to isometry by \cite[Theorem 1.14.4]{Nik79}, we can assume that every summand of $U\oplus 2 E_8(-1)$ is mapped identically onto the corresponding summand of $\mathrm{II}_{2,26}$. Therefore, any choice of a primitive vector $l\in U \oplus E_8(-1)$ of norm $l^2=2k$ gives a primitive embedding
$$L_{2k}=U\oplus 2E_8(-1) \oplus \langle 2k \rangle \hookrightarrow \mathrm{II}_{2,26}.$$

In this section we prove the following:

\begin{prop}\label{prop:ram}
The quasi pullback $\Phi|_{L_{2k}}$ defined in Thereom \ref{thm:quasi} vanishes along the ramification divisor of
$$ \pi: \Omega_{L_{2k}} \rightarrow \cM_{2k}=\widetilde{\Or}^+(L_{2k}) \backslash \Omega_{L_{2k}} $$
for any primitive embedding $ L_{2k} \hookrightarrow \mathrm{II}_{2,26}$ such that $(L_{2k})^\bot_{\mathrm{II}_{2,26}}$ does not contain a copy of $E_8(-1)$.
\end{prop}

For any $l\in L$ we define its \textit{divisibility} $\divi(l)$ to be the unique $m>0$ such that $(l,L)=m \ZZ$ or, equivalently, the unique $m>0$ such that $l/m\in L^\vee$ is primitive. Since $\divi(r)>0$ is the smallest intersection number of $r$ with any other vector, $\divi(r)$ divides $ r^2$. Moreover, if $r$ is reflective, the number $2\frac{(l,r)}{r^2}$ must be an integer, so $r^2$ divides $2(l,r)$ for all $l\in L$, i.e.\ $r^2\mid 2\divi(r)$. Summing up
$$\divi(r)\mid r^2\mid 2\divi(r).$$

\begin{prop}
Let $r\in L_{2k}$ be a reflective vector. Then $\sigma_r$ induces $\pm \id$ in $A_{L_{2k}}$, i.e.\ $\pm\sigma_r \in \widetilde{\Or}(L)$, if and only if $r^2=\pm 2$ or $r^2=\pm 2k$ and $\divi(r) \in \{ k, 2k \}$.
\end{prop}
\begin{proof}
Similar to \cite[Proposition 3.2, Corollary 3.4]{GHS07} .
\end{proof}

Now $\sigma_r\in \Or^+(L\otimes \RR)$ if and only if $r^2<0$ (see \cite{GHS07b}). Recall that an integral lattice $T$ is called \textit{2-elementary} if $A_T$ is an abelian 2-elementary group.

\begin{prop} \label{prop:elementary}
Let $r\in L_{2k}$ be primitive with $r^2=-2k$ and $\divi(r) \in \{ k, 2k \}$. Then $L_r:=r^\bot_{L_{2k}}$ is a $2$-elementary lattice of signature $(2,16)$ and determinant $4$.
\end{prop}
\begin{proof}
We have the following well-known formula for $\det(L_r)$ (see for instance \cite[Equation 20]{GHS07}):
$$\det(L_r)=\frac{\det(L_{2k})\cdot r^2}{\divi(r)^2}\in \{1, 4\}.$$
Since $L_{2k}$ has signature $(2,17)$ and $r^2<0$, we have that $L_r$ has signature $(2,16)$. Therefore $\det(L_r)$ cannot be $1$, because there are no unimodular lattices with signature $(2,16)$ (see \cite[Theorem 0.2.1]{Nik79}). This shows that $\divi(r)=k$. Therefore the reflection $\sigma_r$ acts as $-\id$ on the discriminant group $A_{L_{2k}}$ (see \cite[Corollary 3.4]{GHS07}). Now we can extend $-\sigma_r\in \widetilde{\Or}(L_{2k})$ to an element $\bar{\sigma}_r \in \Or(\Lambda_{K3})$ by defining ${\bar{\sigma}_r|}_{U\oplus \langle -2k \rangle}=\id$ on the orthogonal complement of $L_{2k}\hookrightarrow \Lambda_{K3}$. Put $S_r:=(L_r)^\perp_{\Lambda_{K3}}$. It is easy to see that
$$ {\bar{\sigma}_r|}_{L_r}=-\id \quad  \text{ and } \quad {\bar{\sigma}_r|}_{S_r}=\id.$$
Then $L_r$ is $2$-elementary by \cite[Corollary I.5.2]{Nik79}. 
\end{proof}

\begin{prop} \label{prop:ort}
Given any embedding $L_{2k}\hookrightarrow \mathrm{II}_{2,26}$, let $r\in L_{2k}$ be a primitive reflective vector with $r^2=-2k$, and consider $L_r=r^\bot_{L_{2k}}$ as above. Under the chosen embedding, the orthogonal complement $(L_r)^\bot_{\mathrm{II}_{2,26}}$ is isomorphic to either
$$D_{10}(-1) \text{  or  } E_8(-1)\oplus 2A_1(-1).$$
\end{prop}
\begin{proof}
Since $\mathrm{II}_{2,26}$ is unimodular, the discriminant groups of $L_r$ and $(L_r)^\bot_{\mathrm{II}_{2,26}}$ are isometric up to a sign. Proposition \ref{prop:elementary} thus implies that $(L_r)^\bot_{\mathrm{II}_{2,26}}$ is a $2$-elementary, negative definite lattice of rank $10$ and determinant $4$. By \cite[Proposition 1.8.1]{Nik79}, any $2$-elementary discriminant form is isometric to a direct sum of finite quadratic forms, each of which is isometric to one of four finite quadratic forms, namely the discriminant forms of the $2$-elementary lattices $A_1$, $A_1(-1)$, $U(2)$, $D_4$. Since $(L_r)^\bot_{\mathrm{II}_{2,26}}$ has signature $-2 \pmod{8}$ and determinant $4$, it is immediate to see that its discriminant form must be isometric to the discriminant form of $2A_1(-1)$. Now we notice that the lattice $E_8(-1)\oplus 2A_1(-1)$ is a $2$-elementary, negative definite lattice of rank $10$ with the desired discriminant form. Finally it is enough to compute the genus of $E_8(-1)\oplus 2A_1(-1)$. A quick check with \texttt{Magma} yields that the whole genus consists of $E_8(-1)\oplus 2A_1(-1)$ and $D_{10}(-1)$. Alternatively, one can use the Siegel mass formula \cite{CS88} and check that the mass of the quadratic form $f$ associated to the lattice $E_8(-1)\oplus 2A_1(-1)$ is
$$m(f)=\frac{5}{2^8\cdot 4! \cdot 1814400}=\frac{1}{2229534720}.$$
Since a straightforward check shows that $D_{10}(-1)$ is in the genus of $E_8(-1)\oplus 2A_1(-1)$, and the equality
$$\frac{1}{|\Or(D_{10}(-1))|}+\frac{1}{|\Or(E_8(-1)\oplus 2A_1(-1))|}=\frac{1}{3715891200}+\frac{1}{5573836800}=\frac{1}{2229534720}=m(f)$$
holds, we deduce that $\{D_{10}(-1),E_8(-1)\oplus 2A_1(-1)\}$ is the whole genus of $E_8(-1)\oplus 2A_1(-1)$.
\end{proof}

Now we are ready to prove Proposition \ref{prop:ram}.

 \begin{proof}[Proof of Proposition \ref{prop:ram}]
 In order to prove that $\Phi|_{L_{2k}}$ vanishes along the ramification divisor of the projection $\pi$, we have to show that $\Phi|_{L_{2k}}$ vanishes on the $(-2k)$-divisors $\Omega_r(L_{2k})$ given by reflective vectors $r\in L_{2k}$ of norm $-2k$ (see Theorem \ref{thm:Rdiv}). Hence let $r$ be a $(-2k)$-reflective vector. By Proposition \ref{prop:ort}, $(L_r)^\bot_{\mathrm{II}_{2,26}}$ is a root lattice with at least $180$ roots ($E_8(-1)\oplus 2A_1(-1)$ has $244$ and $D_{10}(-1)$ has $180$). Since by assumption the orthogonal complement of $L_{2k}$ in $\mathrm{II}_{2,26}$ does not contain a copy of $E_8(-1)$, the root lattice generated by $R_{-2}(L^\perp_{2k})$ has rank at most $9$ and does not contain a copy of $E_8(-1)$. By checking all such root lattices, we obtain $|R_{-2}(L^\perp_{2k})| \le |\{\text{roots of }D_9\}|=144$ (just recall that $A_n$ has $n(n+1)$ roots, $D_n$ has $2n(n-1)$ roots, $E_6,E_7$ have $72$ and $126$ roots respectively). Consequently $\Phi|_{L_{2k}}$ vanishes along the $(-2k)$-divisor $\Omega_r(L_{2k})$ given by $r$ with order $\ge (180-144)/2>0$, as claimed. 
 \end{proof}

\section{Lattice engineering}\label{sec:lattice}

By the previous discussion, we have transformed our original question of determining the Kodaira dimension of $\cM_{2k}$ to the following
\begin{prob}\label{prob}
For which $2k>0$ does there exist a primitive vector $l\in U\oplus E_8(-1)$ with norm $l^2=2k$ such that $l$ is orthogonal to at least 2 and at most 8 roots?
\end{prob}

We want to find a lower bound for the values $2k$ answering Problem \ref{prob} positively (see Proposition \ref{prop:answer}). Since $U\oplus E_8(-1)$ contains infinitely many roots, we want to start by reducing to the more manageable case of $E_8(-1)$, whose number of roots is finite.

For simplicity we define
$$ R(l):=\{ r\in U\oplus E_8(-1) \ : \ r^2=-2, (r, l)=0 \}=R_{-2}(L^\perp_{2k}). $$

The following is a slight generalization of \cite[Lemma 4.1,4.3]{TVA19}.

\begin{lemma} \label{lemma:E8}
Let $l=\alpha e+\beta f +v$, where $U=\langle e, f \rangle $ such that $e^2=f^2=0$ and $ef=1$,  $v\in E_8(-1)$ and $\alpha, \beta \in \ZZ$, with norm $l^2=2k>0$. Let $r=\alpha'e+\beta'f+v'$ be a vector of $R(l)$, where $v'\in E_8(-1)$ and $\alpha', \beta' \in \ZZ$. If $\alpha\ne \beta$, $\alpha,\beta>\sqrt{k}$ and $\alpha\beta<\frac{5}{4}k$, then $\alpha'=\beta'=0$.
\end{lemma}
\begin{proof}
See \cite[Lemma 3.3 and 3.4]{Pet19}.
\end{proof}

In other words, if $l=\alpha e+\beta f+v\in U\oplus E_8(-1)$ is a vector of norm $2k$ satisfying the assumptions of Lemma \ref{lemma:E8}, then the roots of $U\oplus E_8(-1)$ orthogonal to $l$ are roots of $E_8(-1)$. Therefore the set $R(l)$ coincides with the set of roots in $v^\bot_{E_8(-1)}$. The following lemma, inspired by \cite[Theorem 7.1]{GHS07}, controls the number of roots of $E_8(-1)$ orthogonal to $v$.

\begin{lemma} \label{lemma:hulek}
There exists $v\in E_8$ with $v^2=2n$ and such that $v^\bot_{E_8}$ contains at least $2$ and at most $8$ roots if the inequality
\begin{equation} \label{eq:inequality}
    2N_{E_7}(2n) > 28N_{E_6}(2n)+63N_{D_6}(2n),
\end{equation}
holds, where $N_L(2n)$ denotes the number of representations of $2n$ by the positive definite lattice $L$.
\end{lemma}
\begin{proof}
We follow closely \cite[Theorem 7.1]{GHS07}. Let $a\in E_8$ be a root. Its orthogonal complement $E_7^{(a)}:=a^\bot_{E_8}$ is isometric to $E_7$. The set of $240$ roots in $E_8$ consists of the $126$ roots in $E_7^{(a)}$ and $114$ other roots, forming the subset $X_{114}$. Assume that every $v\in E_7^{(a)}$ with $v^2=2n$ is orthogonal to at least $10$ roots in $E_8$, including $\pm a$. By \cite[Lemma 7.2]{GHS07} we know that every such $v$ is contained in the union
\begin{equation} \label{eq:N}
    \bigcup_{i=1}^{28}{(A_2^{(i)})^\bot_{E_8}} \sqcup \bigcup_{j=1}^{63}{(A_1^{(j)})^\bot_{E_7^{(a)}}},
\end{equation}
where $A_2^{(i)}$ (resp. $A_1^{(j)}$) are root systems of type $A_2$ (resp. $A_1$) contained in $X_{114}$ (resp. $E_7^{(a)}$). Denote by $n(v)$ the number of components in the union (\ref{eq:N}) containing $v$. Since $(A_2^{(i)})^\bot_{E_8}\cong E_6$ and $(A_1^{(j)})^\bot_{E_7^{(a)}}\cong D_6$, we have counted the vector $v$ exactly $n(v)$ times in the sum
$$28N_{E_6}(2n)+63N_{D_6}(2n).$$
We distinguish three cases.
\begin{enumerate}[(i)]
    \item If $v\cdot c \ne 0$ for every $c\in X_{114}\smallsetminus \{\pm a\}$, then $v$ is orthogonal to at least $4$ copies of $A_1$ in $E_7^{(a)}$, so $n(v)\ge 4$.
    \item If $v$ is orthogonal to only one $A_2^{(i)}$ ($6$ roots), then $v$ is orthogonal to at least $2$ copies of $A_1$ in $E_7^{(a)}$, so $n(v)\ge 3$.
    \item If $v$ is orthogonal to at least two $A_2^{(i)}$, then $n(v)\ge 2$.
\end{enumerate}
In conclusion $n(v)\ge 2$ for every $v\in E_7^{(a)}$. Therefore, under our assumption that every $v\in E_7^{(a)}$ with $v^2=2n$ is orthogonal to at least $10$ roots, we have shown that any such $v$ is contained in at least $2$ sets of the union (\ref{eq:N}), i.e.\
$$2N_{E_7}(2n) \le 28N_{E_6}(2n)+63N_{D_6}(2n).$$
\end{proof}

\begin{prop} \label{prop:952}
Let $n\ge 952$. Then there exists $v\in E_8(-1)$ with $v^2=-2n$ such that $v^\bot_{E_8(-1)}$ contains at least $2$ and at most $8$ roots.
\end{prop}
\begin{proof}
\cite[Equations (31), (33) and (34)]{GHS07} give the following estimates:
$$N_{E_7}(2n)>123.8 \ n^{5/2}, \qquad N_{E_6}(2n)<103.69 \ n^2,\qquad N_{D_6}(2n)<75.13 \ n^2.$$
By Lemma \ref{lemma:hulek}, we immediately obtain the claim.
\end{proof}

We are now ready to answer Problem \ref{prob}:

\begin{prop}\label{prop:answer}
Let $k\ge 4900$. Then there exists a primitive $l\in U\oplus E_8(-1)$ with $l^2=2k$ and $2\le |R(l)|\le 8$.
\end{prop}
\begin{proof}
Pick $k>0$ and consider $l=\alpha e + \beta f + v$, where $l^2=2k$, $v^2=-2n$, so that $\alpha\beta=n+k$. Suppose that there exist $\alpha$ and $\beta$ satisfying the hypotheses of Lemma \ref{lemma:E8} such that $n=\alpha\beta-k \ge 952$. Then Proposition \ref{prop:952} implies that we can find a $v\in E_8(-1)$ with $v^2=-2n$ such that $v^\bot_{E_8(-1)}$ contains at least $2$ and at most $8$ roots. Moreover Lemma \ref{lemma:E8} ensures that the roots of $U\oplus E_8(-1)$ orthogonal to $l=\alpha e +\beta f +v$ are contained in $E_8(-1)$, so that $l^\bot_{U\oplus E_8(-1)}$ also contains at least $2$ and at most $8$ roots. Therefore the existence of such $\alpha,\beta$ is sufficient for the existence of $l\in U\oplus E_8(-1)$ with $2\le |R(l)|\le 8$.

Now let $k\ge 4900=70^2$, and consider
$$\alpha=\lceil \sqrt{k}+6\rceil,\qquad \beta=\alpha+1.$$
Clearly $\alpha\ne \beta$, $\gcd(\alpha,\beta)=1$ and $\alpha,\beta > \sqrt{k}$. Moreover
$$\frac{5}{4}k-\alpha\beta \ge \frac{5}{4}k - (\sqrt{k}+7)(\sqrt{k}+8)=\frac{1}{4}k-15\sqrt{k}-56>0 ,$$ 
and
$$n=\alpha\beta-k \ge (\sqrt{k}+6)(\sqrt{k}+7)-k= 13\sqrt{k}+42\ge 952,$$
completing the proof.
\end{proof}

In order to deal with the remaining values of $k$, we use of the geometry of the K3 surfaces with N\'eron-Severi lattice isometric to $U\oplus E_8(-1)$. We recall in the  following the main properties of such surfaces.\newline

Let $X$ be a K3 surface with $\NS(X)=U\oplus E_8(-1)$. Then $X$ has finite automorphism group and a finite number of irreducible $(-2)$-curves (see e.g. \cite{Nik79a} or \cite{Kon89}). More precisely, if $|E|$ denotes the unique elliptic fibration on $X$, then the irreducible $(-2)$-curves on $X$ are the $9$ curves $C_2,\ldots,C_{10}$ contained in the reducible fiber of $|E|$, plus the unique section of $E$, which we will denote by $C_{1}$. The dual graph of such $(-2)$-curves is
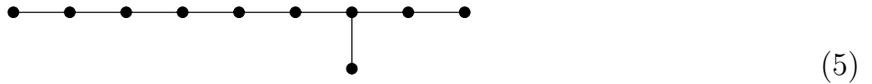
\begin{equation} \label{eq:diagram}
    \begin{tikzpicture}
  \filldraw 
(-2.25,0) circle (2pt) -- (-1.5,0) circle (2pt)  -- (-0.75,0) circle (2pt)  -- (0,0) circle (2pt)  -- (0.75,0) circle (2pt)  -- (1.5,0) circle (2pt)  -- (2.25,0) circle (2pt)  -- (3,0) circle (2pt)  -- (3.75,0) circle (2pt) 
(2.25,0) -- (2.25,-0.75) circle (2pt) ;
  \end{tikzpicture}
\end{equation}
where $C_1,\ldots,C_7,C_9,C_{10}$ are the curves in the upper line, and $C_8$ is such that $C_7C_8=1$.\newline
  
  Now let $D\in \NS(X)=U\oplus E_8(-1)$ be a primitive divisor of norm $2k>0$ with $2 \le |R(D)| \le 8$. In other words, $D^\perp$ contains at least $1$ and at most $4$ effective $(-2)$-divisors. Up to the action of the Weyl group $W < \Or(U\oplus E_8(-1))$ we can assume that $D$ is nef, since isometries of $U \oplus E_8(-1)$ do not change the number of orthogonal roots. The nef cone $\mathrm{Nef}(X)$ is rational polyhedral, and a basis can be computed in \texttt{Magma}. It turns out that such basis $\{D_1,\ldots,D_{10}\}$ is the dual basis of $\{C_1,\ldots,C_{10}\}$, i.e.\ $D_i C_j =\delta_{ij}$ for all $1 \le i,j \le 10$. For instance, $D_1=E$ defines the only elliptic fibration on $X$, and $D_1^2=0$. Moreover $D_i^2>0$ for $i\ge 2$. This means that any nef divisor $D$ is a linear combination of $D_1,\ldots,D_{10}$ with non-negative coefficients
  $$D=\sum_{i=1}^{10}{d_i D_i}.$$
  By construction, the $(-2)$-curve $C_j$ is orthogonal to $D$ if and only if $d_j=0$. This implies that the root part of $D^\perp$ is a root lattice $R$ generated by the $(-2)$-curves $\{C_j \mid d_j=0\}$. Since $R$ contains at most $4$ effective roots, it is one of the following root lattices:
  $$A_1(-1), \ 2A_1(-1), \ 3A_1(-1), \ 4A_1(-1), \ A_2(-1), \ A_2(-1)\oplus A_1(-1).$$
  Fix one of the finitely many sub-diagrams $J \subseteq \{1,\ldots,10\}$ of the dual graph (\ref{eq:diagram}) giving rise to a root lattice $\langle C_j \mid j \in J\rangle$ isometric to $R$. Thus the nef divisors $D$ orthogonal precisely to $\{C_j \mid j \in J\}$ are all those of the form
  $$D=\sum_{i \notin J}{d_i D_i}$$
  for some $d_i > 0$. Since we are only interested in divisors of norm $2k < 2\cdot 4900$, we can use the inequality
$$D^2 \ge \sum_{ i \notin J}{d_i ^2 D_i^2} + 2 \sum_{1 \ne i \notin J}{d_1d_i (D_1D_i)},$$
to bound the $d_i$ for $i \notin J$. More precisely, we have that $$d_i^2 \le \frac{2 \cdot 4900}{D_i^2} \ \forall i \ge 2, \quad \mathrm{and} \quad d_1 \le \frac{4900}{\sum_{1 \ne i \notin J}{D_1D_i}}.$$
 
  By varying the coefficients $d_i$'s in these ranges, we obtain all primitive vectors $D \in U \oplus E_8(-1)$ with $D^2 \le 2 \cdot 4900$ and $2 \le |R(D)| \le 8$ up to the action of $\Or(U\oplus E_8(-1))$. Therefore this search is completely exhaustive.
  
  A similar list can be obtained if we allow $D$ to have up to $10$ orthogonal roots. All the previous discussion works analogously, with the only difference that the root part of $D^\perp$ can also be isometric to $5A_1(-1)$ or $A_2(-1)\oplus 2A_1(-1)$. 
  
  We get the following: a primitive vector $l \in U \oplus E_8(-1)$ with $l^2=2k<2 \cdot 4900$ and $2 \le |R(l)|\le 8$ exists if and only if
  \begin{equation} \label{eq:set1}
 k \ge 208, \  k\ne 211,219 \text{ or }  k\in \{170,185,186,188,190,194,200,202,204,206\}.
\end{equation}
Moreover a similar vector $l$ with $2\le |R(l)|\le 10$ exists if and only if
\begin{equation} \label{eq:set2}
 k\ge 164,  \ k \ne 169,171,175 \text{ or } k\in \{140,146,150,152,154,155,158,160,162\}.
\end{equation}
  
 We have implemented the algorithm described above in \texttt{Magma}. The interested reader can find it in the \texttt{Arxiv} \footnote{arXiv:2003.10957} distribution of this paper, together with a list of primitive vectors realizing the values of $k$ listed in (\ref{eq:set1}) and (\ref{eq:set2}).

We are now ready to prove Theorem \ref{thm:main}.

\begin{proof}[Proof of Theorem \ref{thm:main}]
Proposition \ref{prop:answer} combined with the subsequent search ensures that there exists a primitive $l\in U\oplus E_8(-1)$ with norm $l^2=2k$ and $2\le |R(l)|\le 8$ if $k\ge 4900$ or $k$ belongs to the list (\ref{eq:set1}), in particular for any $k\geq220$. Such an $l\in U\oplus E_8(-1)$ determines an embedding $L_{2k}\hookrightarrow \mathrm{II}_{2,26}$ with the property
$$ 1 \leq N(L_{2k}) \leq 4, $$
where $N(L_{2k})$ is the number of effective roots in the orthogonal complement $(L_{2k})^\perp_{\mathrm{II}_{2,26}} $. Hence Theorem \ref{thm:quasi} provides a non-zero cusp form $\Phi|_{L_{2k}}$ of weight $12+N(L_{2k})\leq 12+4<17=\dim(\cM_{2k})$, which vanishes along the ramification divisor of $\pi: \Omega_{L_{2k}} \rightarrow \cM_{2k}$ in view of Proposition \ref{prop:ram}, since $l^\perp$ does not contain $E_8(-1)$, otherwise $l$ would be orthogonal to at least $240$ roots. Then the low-weight cusp form trick (Theorem \ref{thm:trick}) ensures that $\cM_{2k}$ is of general type. 

An analogous argument shows that $\cM_{2k}$ has non-negative Kodaira dimension if $k$ belongs to the list (\ref{eq:set2}), in particular for any $k\geq 176$.
\end{proof}

\section{Geometric constructions}\label{sec:constructions}

In this section we recall three well-known geometric constructions of K3 surfaces. Namely, double covers of the quadric surface $\PP^1\times \PP^1$ (see \S \ref{subsec:F_0}) and of the Hirzebruch surface $\FF_4$ (see \S \ref{subsec:F_4}) branched over suitable curves define lattice polarized K3 surfaces with respect to the lattices $U(2)$ and $U$ respectively. Furthermore, every elliptic K3 surface can be reconstructed from its Weierstrass fibration (see \S \ref{subsec:Weierstrass}).

\subsection{Double covers of $\PP^1\times \PP^1$} \label{subsec:F_0} Let $\FF_0:=\PP^1\times \PP^1$ be the smooth quadric surface in $\PP^3$. Its Picard group is generated by the classes of the two pencils $\ell_1,\ell_2$ of lines, hence $\Pic(\FF_0)$ endowed with the intersection form on $\FF_0$ is isomorphic to the hyperbolic plane $U$. The canonical bundle is $K_{\FF_0}=\cO_{\FF_0}(-2,-2)$.

Now let $\pi:X\rightarrow \FF_0$ be the double cover branched over a smooth curve $B\in |-2K_{\FF_0}|=|\cO_{\FF_0}(4,4)|$. Then $X$ is a smooth K3 surface. The pullbacks $E_i=\pi^*\ell_i$ for $i=1,2$ are smooth elliptic curves, and $E_1E_2=2\ell_1\ell_2=2$, so that
$$\langle E_1,E_2\rangle =U(2) \hookrightarrow \NS(X).$$
This embedding is primitive, and $\NS(X)=U(2)$ for a very general branch divisor $B$.

Assume now that there exists a smooth rational curve $C\in |\cO_{\FF_0}(1,d)|$ for $d\ge 0$ intersecting $B$ with even multiplicities. For instance, $C$ can be simply tangent to $B$ in exactly $2d+2$ points. Then we have the following (cf. \cite[Proposition 5.1]{Fes18}):

\begin{lemma} \label{lemma:even}
Let $\nu:X\rightarrow Y$ be a double cover of smooth projective surfaces branched over a smooth curve $B$, and assume that there exists a smooth rational curve $C\subseteq Y$ intersecting $B$ with even multiplicities. Then the pullback $\nu^*C$ splits into two disjoint irreducible components, both isomorphic to $C$.
\end{lemma}
\begin{proof}
Let $D:=\nu^{-1}(C)\subseteq X$. The double cover $\nu$ induces a double cover $\overline{\nu}:D\rightarrow C$, which is isomorphic to an unbranched double cover. This is because the branch locus of $\overline{\nu}$ coincides with the set $b(C):=\{x\in C \mid \mathrm{mult}_x(C,B)\equiv 1 \pmod{2}\}=\varnothing$. The unique unbranched double cover of $C\cong \PP^1$ is given by a disjoint union of two smooth rational curves isomorphic to $C$.
\end{proof}

In the case $Y=\FF_0$ as above, the pullback $D=\pi^*C=D_1+D_2$ splits into the union of two irreducible components $D_1,D_2\cong \PP^1$. Since $D_1$ is smooth rational, we have $D_1^2=-2$, and moreover $D_1E_1=1$, $D_1E_2=d$. This implies that there exists an embedding (not necessarily primitive)
$$\langle E_1,E_2,D_1\rangle = \begin{pmatrix}
0 &2 &1\\
2 &0 &d\\
1 &d &-2
\end{pmatrix}\cong U\oplus \langle -2(2d+4)\rangle \hookrightarrow \NS(X).$$

If instead the branch divisor $B$ is not smooth, but has simple singularities, the double cover $\pi:X\rightarrow \FF_0$ is a K3 surface with isolated simple singularities. Therefore the minimal desingularization $\widetilde{X}\rightarrow X$ is a smooth K3 surface, since simple singularities do not change adjunction.

The following result is well known, but we include its proof for the sake of completeness.

\begin{prop} \label{prop:F0}
Let $X$ be an elliptic K3 surface with $\NS(X)\cong U \oplus \langle -2k\rangle$ for some $k\geq 1$. Then $X$ can be realized as a double cover of $\FF_0$ if and only if $k$ is even and $k\geq 4$.
\end{prop}
\begin{proof}
If $X$ is a double cover of $\FF_0$, the pullback map induces a primitive embedding
$$U(2)\hookrightarrow \NS(X)=U \oplus \langle -2k\rangle.$$
Any even lattice of rank $3$ containing primitively $U(2)$ has discriminant divisible by $4$, so we conclude that $k=\frac{1}{2}\det(\NS(X))$ is even.

Conversely assume that $\NS(X)=U\oplus \langle -2k\rangle$ for a certain $k\ge 4$ even. Then as above we have an isomorphism
$$\begin{pmatrix}
0 &2 &1\\
2 &0 &d\\
1 &d &-2
\end{pmatrix}\cong U\oplus \langle -2k\rangle $$
for $d=\frac{1}{2}(k-4)\ge 0$, so there are two genus one fibrations $|E_1|,|E_2|:X\rightarrow \PP^1$ induced by the two elements $E_1,E_2$ of the basis of square zero. We can now consider the surjective map $$\pi=(|E_1|,|E_2|):X\rightarrow \FF_0.$$ 
It is a morphism of degree $2$, since the preimage of any point of $\FF_0$ consists of the two points of intersection of two elliptic curves in $|E_1|$ and $|E_2|$, as $E_1E_2=2$. Consider the branch divisor $B$; if $B$ is smooth, then $\pi$ is a double cover, as claimed. Assume by contradiction that $B$ is singular. $B$ must have simple singularities, since otherwise the canonical divisor of $X$ would be strictly negative. Thus $X$ is the desingularization of the double cover $\widetilde{\pi}:\widetilde{X}\rightarrow \FF_0$ branched over $B$, and therefore $\NS(X)$ contains the class of a smooth rational curve orthogonal to $U$. This is however absurd, since $\rk{\NS(X)}=3$ and $\NS(X)\not\cong U\oplus A_1(-1)$.

It only remains to deal with the case $k=2$, so consider a K3 surface $X$ with $\NS(X)=U\oplus \langle -4\rangle$. If by contradiction $X$ is a double cover of $\FF_0$, then $\NS(X)$ contains primitively $U(2)$, so that
$$U\oplus \langle -4\rangle \cong \begin{pmatrix}
0 &2 &a\\
2 &0 &b\\
a &b &-2c
\end{pmatrix}$$
for $a,b,c\in \ZZ$, $c\ge 1$. Say that this isomorphism is given by the choice of a basis $\{E_1,E_2,D\}$. The determinant of $\NS(X)$ is $4$, and this forces $ab+2c=1$. Thus $a,b$ are odd, and without loss of generality $a<0$, $b>0$. Now choose $n\ge 0$ such that $a+2n=1$ and consider the divisor $D+nE_2$. It is effective by Riemann-Roch, since
$$(D+nE_2)^2=-2c+2nb=-2c+b(1-a)=-2c-ab+b=b-1\ge 0$$
and $D+nE_2$ has intersection $1\ge 0$ with the nef divisor $E_1$. Moreover $(D+nE_2)E_1=1$ means that $D+nE_2$ coincides with $kE_1+S$ for a certain $k\ge 0$ and a section $S$ of the elliptic pencil $|E_1|$. In other words, $\NS(X)$ is generated by the three elements $E_1,E_2,S$. However the intersection form of $X$ with respect to this basis is
$$\begin{pmatrix}
0 &2 &1\\
2 &0 &\alpha\\
1 &\alpha &-2
\end{pmatrix}$$
and this matrix has determinant $4$ only if $\alpha=-1$, which is a contradiction, as $E_2$ is nef and $S$ is effective.
\end{proof}

\begin{rmk} \label{oss:all}
Let $X$ be a K3 surface with $\NS(X)=U\oplus \langle -2k\rangle$ for a certain $k\ge 4$ even. Then an argument as above shows that a basis of $\NS(X)$ is given by $\{E_1,E_2,D\}$ with intersection matrix
$$ \begin{pmatrix}
0 &2 &1\\
2 &0 &d\\
1 &d &-2
\end{pmatrix}$$
where $d=\frac{1}{2}(k-4)$, $\pi=(|E_1|,|E_2|):X\rightarrow \FF_0$ is the double cover branched over a $(4,4)$-curve $B$, and $C=\pi(D)$ is a smooth $(1,d)$-curve meeting $B$ with even multiplicities.
\end{rmk}

\subsection{Double covers of $\FF_4$}
Consider the \textit{Hirzebruch surface} \label{subsec:F_4} $\FF_4:=\PP(\cO_{\PP^1}\oplus \cO_{\PP^1}(4))$. We denote by $p:\FF_4\rightarrow \PP^1$ the $\PP^1$-bundle structure. We have that $\Pic(\FF_4)=\ZZ \langle f,s\rangle$, where $f$ is the class of a fiber $F$ of the projection $p$, while $s$ is the class of the unique curve $S\subseteq \FF_4$ with negative self-intersection. The intersection form on $\Pic(\FF_4)$ with respect to this basis is
$$\begin{pmatrix}
0 &1 \\
1 &-4 
\end{pmatrix}\cong U.$$
The canonical bundle of $\FF_4$ is given by $K_{\FF_4}=-2s-6f$. Notice that $\varphi=\varphi_{|s+4f|}:\FF_4\rightarrow C_4$ is the desingularization of the quartic cone $C_4\subseteq \PP^5$ over the normal rational curve $C=\Img(|\cO_{\PP^1}(4)|)\subseteq \PP^4$.

Now consider the double cover $\pi:X\rightarrow \FF_4$ branched over a curve $B\in |-2K_{F_4}|=|4s+12f|$. The linear system $|4s+12f|$ has a fixed part, given by the curve $S$, and a moving part $|3s+12f|$. Assume that $B$ splits as the sum $S+B_0$, where $B_0\in |3s+12f|$ is a smooth irreducible curve disjoint from $S$, as $s(3s+12f)=0$. Then the surface $X$ is a smooth K3 surface. The pullback $E=\pi^*F$ is a smooth elliptic curve, since the restricted double cover $E\rightarrow F$ is branched over $(4s+12f)f=4$ points. Moreover $\pi$ is totally ramified over $S\subseteq B$, so $\pi^*S=2C$, where $C=\pi^{-1}(S)\cong \PP^1$ is a smooth rational curve. Since $EC=\frac{1}{2}(\pi^*F)(\pi^*S)=FS=1$, we have a primitive embedding
$$\begin{pmatrix}
0 &1 \\
1 &-2 
\end{pmatrix}\cong U\hookrightarrow \NS(X).$$
For a very general branch divisor $B$, we simply have $\NS(X)\cong U$.

Consider the linear system $|s+2kf|$ for $k\ge 2$. Its general member $D$ is a smooth rational curve meeting $F$ in $1$ point, $S$ in $2k-4$ points and $B$ in $(s+2kf)(4s+12f)=8k-4$ points. Assume further that the curve $D$ intersects the branch divisor $B$ with even multiplicities. Then Lemma \ref{lemma:even} ensures that the pullback $\pi^*D=D_1+D_2$ splits into two disjoint components $D_1,D_2\cong \PP^1$. This implies that there exists an embedding (not necessarily primitive)
$$\langle E,C,D_1\rangle = \begin{pmatrix}
0 &1 &1\\
1 &-2 &k-2\\
1 &k-2 &-2
\end{pmatrix}\cong U\oplus \langle -2k\rangle \hookrightarrow \NS(X),$$
since $D_1E=\frac{1}{2}(\pi^*D)(\pi^*F)=DF=1$ and $D_1C=\frac{1}{4}(\pi^*D)(\pi^*S)=\frac{1}{2}DS=k-2$.

\begin{prop}
Every elliptic K3 surface $X$ is the desingularization of a double cover of the Hirzebruch surface $\FF_4$.
\end{prop}
\begin{proof}
Assume that $U\hookrightarrow \NS(X)$, and denote by $E,C$ the smooth curves in $X$ generating $U$ such that $E^2=0$, $C^2=-2$. Consider the linear system $|4E+2C|$. By \cite[Corollary 8.1.6]{Huy16} the divisor $4E+2C$ is nef, as it has non-negative intersection with every smooth rational curve. Moreover $4E+2C$ has intersection $0$ with the curve $C$. Since $(4E+2C)^2=8$ and $\dim{|4E+2C|}=5$, $\psi=\varphi_{|4E+2C|}:X\rightarrow \PP^5$ is a morphism onto a surface $Y\subseteq \PP^5$ contracting $C$. $C$ is a smooth $(-2)$-curve, so $Y$ is singular. Now the elliptic curve $E$ has intersection $(4E+2C)E=2$ with $4E+2C$, so $\psi$ has degree $2$ by \cite[Theorem 5.2]{Sai74}. This implies that $\deg(Y)=4$, so $Y\subseteq \PP^5$ is a singular surface of minimal degree, hence $Y$ is the quartic cone $C_4$ (see \cite{dPe87}). Therefore $\psi$ factors through the minimal resolution of $C_4$, which is $\FF_4$, giving a morphism $\pi:X\rightarrow \FF_4$ of degree $2$. Now we can repeat the argument in the proof of Proposition \ref{prop:F0}, obtaining that $X$ is the desingularization of a double cover of $\FF_4$.
\end{proof}

\begin{rmk}
Every K3 surface $X$ with $\NS(X)=U\oplus \langle -2k\rangle$ for a certain $k\ge 2$ can be obtained as a double cover $\pi:X\rightarrow \FF_4$ branched over a smooth curve $B\in |4s+12f|$ admitting a rational curve $D\in |s+2kf|$ intersecting $B$ with even multiplicities. 

If instead $X$ is a K3 surface with $\NS(X)=U\oplus \langle -2\rangle$, then it is the desingularization of the double cover of $\FF_4$ branched over a curve $B$ with a unique singularity of type $A_1$.
\end{rmk}

\subsection{Weierstrass fibrations}\label{subsec:Weierstrass}
Let $X$ be a smooth K3 surface. Recall that $X$ is called \textit{elliptic} if it admits an elliptic fibration, i.e.\ a morphism $\pi: X\rightarrow \PP^1$ whose general fiber is a curve of genus one, together with a distinguished section. The N\'eron-Severi group of an elliptic K3 surface contains primitively a copy of the hyperbolic plane $U$, spanned by the classes of the fiber and the zero section of the elliptic fibration. 

Let $X$ be a smooth elliptic K3 surface. By \cite[Section II.3]{Mir89} $X$ is the desingularization of a \textit{Weierstrass fibration} $\pi':Y\rightarrow \PP^1$, where $Y$ is defined by an equation
\begin{equation} \label{eq:projweierstrass}
    Y^2Z=X^3+AXZ^2+BZ^3
\end{equation}
in $\PP(\cO_{\PP^1}(4)\oplus \cO_{\PP^1}(6)\oplus \cO_{\PP^1})$ with $A\in H^0(\cO_{\PP^1}(8))$ and $B\in H^0(\cO_{\PP^1}(12))$ minimal and with $\Delta=4A^3+27B^2$ not identically zero. Conversely, every such Weierstrass fibration desingularizes to a smooth elliptic K3 surface. We will usually restrict to the chart $\{Z \ne 0\}$ over the affine base $\Aff^1_t\subseteq \PP^1$, where the equation (\ref{eq:projweierstrass}) becomes
\begin{equation} \label{eq:weierstrass}
  y^2=x^3+A(t)x+B(t),  
\end{equation}
with $A$ and $B$ polynomials in $t$ of degree at most $8$ and $12$ respectively. Notice that this is the equation of the generic fiber of the Weierstrass fibration, which is an elliptic curve over $\CC(t)$. Under this identification, sections of the fibration $\pi$ (or $\pi'$) correspond to $\CC(t)$-rational points of equation (\ref{eq:weierstrass}). In particular the distinguished zero section is located at the point at infinity $S_0=(0:1:0)$. Moreover we will write $S=(u(t),v(t))$ to denote the section $S$ of $\pi$ corresponding to the $\CC(t)$-rational point $(u(t),v(t))$ of equation (\ref{eq:weierstrass}). By the above description, $u,v\in \CC(t)$ are rational functions of degree at most $4,6$ respectively. 

\begin{rmk} \label{rmk:weiestrass}
Let $X$ be a $U\oplus \langle -2k \rangle$-polarized K3 surface. If $k \ge 2$, the given elliptic fibration on $X$ admits an extra section $S$ such that $S S_0=k-2$. This follows from the isomorphism of lattices
$$U\oplus \langle -2k \rangle \cong \begin{pmatrix}
0 &1 &1\\
1 &-2 &k-2\\
1 &k-2 &-2
\end{pmatrix}.$$
Conversely, if $X$ is an elliptic K3 surface and $S$ is an extra section with $SS_0=k-2$, then there exists an embedding
$$U\oplus \langle -2k \rangle \hookrightarrow \NS(X).$$
This embedding is not necessarily primitive. However, it is primitive if the lattice $U\oplus \langle -2k\rangle$ has no non-trivial overlattices (for instance if $2k$ is square-free, cf. \cite[Proposition 1.4.1]{Nik79}).

\end{rmk}

\section{Unirationality of $\cM_{2k}$ for small $k$}\label{sec:unirational}

The aim of this section is to prove Theorem \ref{thm:unirational}, i.e.\ the unirationality of $\cM_{2k}$ for $k < 11$ and $k\in \{13,$ $16,$ $17,$ $19,$ $21,$ $25,$ $26,$ $29,$ $31,$ $34,$ $36,$ $37,$ $39,$ $41,$ $43,$ $49,$ $59,$ $61,$ $64\}$. For some of the cases we will use the geometric constructions of Section \ref{sec:constructions}. For the others, we will find projective models of $U\oplus \langle -2k\rangle$-polarized K3 surfaces given by (quasi-)polarizations of degree $\le 8$. More precisely, the strategy will consist in finding $\ZZ$-bases of $U\oplus \langle -2k\rangle$ given by the (quasi-)polarization and $(-2)$-curves of small degree. 

The unirationality of $\cM_{56}$ was proved in \cite{BH17} using relative canonical resolutions.

\subsection{$k=1$}
The variety $\cM_2$ is the moduli space of $U\oplus \langle -2 \rangle$-polarized K3 surfaces. If $X$ is a general K3 surface in $\cM_2$, then $X$ is the desingularization of a Weierstrass fibration $Y$ with an $A_1 $ singularity. Hence $X$ admits an elliptic fibration with a unique reducible fiber, consisting of two irreducible smooth rational curves. A quick inspection of the Kodaira fibers \cite[Table I.4.1]{Mir89} yields that this reducible fiber can be either of type $I_2$ (two smooth rational curves meeting transversely at two distinct points) or \textit{III} (two smooth rational curves simply tangent at one point). This depends on whether the $A_1$ singularity on $Y$ belongs to a nodal or cuspidal rational curve respectively. After moving the singular fiber to $t=0$, $Y$ can be written as a Weierstrass fibration
$$y^2=x^3+a(t)x^2+b(t)x+c(t)$$
satisfying $t \mid b(t)$ and $t^2 \mid c(t)$. Up to a change of coordinates in $x$, this equation is equivalent to the one in (\ref{eq:weierstrass}). Conversely, a general such Weierstrass equation desingularizes to an elliptic K3 surface with an $I_2$ or a \textit{III} fiber at $t=0$. From this description we can define a dominant rational  map
$$\cP_2:=\{(a,b,c)\in H^0(\cO_{\PP^1}(4)) \times H^0(\cO_{\PP^1}(8))\times H^0(\cO_{\PP^1}(12)) : t \mid b(t),\ t^2 \mid c(t)\} \dashrightarrow \cM_2$$
sending the polynomials $(a,b,c)$ into the isomorphism class of the desingularization of the corresponding Weierstrass equation. Since $\cP_2$ is an affine space, $\cM_2$ is unirational.

\subsection{$k=2$}
An $U\oplus \langle -4\rangle$-polarized K3 surface $X$ is an elliptic K3 surface admitting a section $S$ disjoint from the zero section $S_0$ of the given elliptic fibration by Remark \ref{rmk:weiestrass}. Let
$$y^2=x^3+a(t)x^2+b(t)x+c(t)$$
be a Weierstrass equation for $X$, where the point at infinity $S_0=(0:1:0)$ is the zero section. Let $S=(u(t),v(t))$ be the extra section. Notice that the points of intersection of $S$ and $S_0$ coincide with the poles of $v$ (or equivalently of $u$), as $(u(t_0):v(t_0):1)=(0:1:0)$ if and only if $t_0$ is a pole for $v$. But $S$ and $S_0$ are disjoint by assumption, so $u,v$ are simply polynomials of degree at most $4,6$ respectively. After the change of variables $x\mapsto x-u$, $y\mapsto y-v$, the Weierstrass equation becomes
\begin{equation} \label{eq:w}
    y^2+2v(t)y=x^3+d(t)x^2+e(t)x,
\end{equation}
for polynomials $d,e,v$ of degree at most $4,8,6$ respectively. Conversely, a general Weierstrass equation as in (\ref{eq:w}) defines an elliptic K3 surface containing the disjoint sections $S_0=(0:1:0)$, $S=(0,0)$, and therefore an $U\oplus \langle -4\rangle$-polarized K3 surface. This implies that there exists a dominant rational map
$$\cP_4:=\{(d,e,v)\in H^0(\cO_{\PP^1}(4)) \times H^0(\cO_{\PP^1}(8))\times H^0(\cO_{\PP^1}(6)) \} \dashrightarrow \cM_4.$$
$\cP_4$ is an affine space, so $\cM_4$ is unirational.

\subsection{$k=3$}
Let $X$ be the desingularization of a double cover of $\PP^2$ branched over a sextic $B$ with an $A_2$ singularity. Then $X$ is a K3 surface with
$$\langle 2 \rangle \oplus A_2(-1) \cong U \oplus \langle -6\rangle \hookrightarrow \NS(X).$$
Since $6$ is square-free, this embedding is primitive. Conversely, if $X$ is a K3 surface with $\NS(X)\cong \langle 2 \rangle \oplus A_2(-1)$, the linear system associated to the first element of the basis induces a morphism $X\rightarrow \PP^2$ of degree $2$ contracting the two $(-2)$-curves, so $X$ is the desingularization of a double cover of $\PP^2$ branched over a sextic with an $A_2$ singularity. Up to a projective transformation, we can assume that the sextic $B\subseteq \PP^2$ has an $A_2$ singularity at $P=(0:0:1)\in \PP^2$, and that the unique line of $\PP^2$ meeting $B$ in $P$ with multiplicity $3$ is $V(x_0)$. This forces $B$ to be given by an equation $f(x_0,x_1,x_2)\in H^0(\cO_{\PP^2}(6))$ with the coefficients of $x_2^6$, of $x_0x_2^5$, of $x_1x_2^5$, of $x_0x_1x_2^4$ and of $x_1^2x_2^4$ all being zero. We denote by $\cP_6$ the linear subspace of $H^0(\cO_{\PP^2}(6))$ consisting of all such polynomials. Therefore there exists a dominant rational map
$$\cP_6\dashrightarrow \cM_6.$$
$\cP_6$ is an affine space, hence $\cM_6$ is unirational.

\subsection{$k=4$}
By Proposition \ref{prop:F0} and Remark \ref{oss:all} a general $U\oplus\langle -8\rangle$-polarized K3 surface $X$ is the double cover of $\FF_0$ branched over a smooth $(4,4)$-curve $B$ admitting a line $L$ simply tangent to $B$ at $2$ points. Choose coordinates $((x_0:x_1),(y_0:y_1))$ on $\FF_0=\PP^1\times \PP^1$ such that $L=V(x_0)$. If $B$ is given by a bihomogeneous polynomial $f(x_0,x_1,y_0,y_1)$ of bidegree $(4,4)$, then $B$ is tangent to $L$ at $2$ points if and only if
$$f(x_0,x_1,y_0,y_1)=x_0g(x_0,x_1,y_0,y_1)+x_1^4h_1(y_0,y_1)^2h_2(y_0,y_1)^2$$
for $g\in H^0(\cO_{\FF_0}(3,4))$, $h_1,h_2\in H^0(\cO_{\PP^1}(1))$. Therefore we get a dominant rational map
$$\cP_8:=\{(g,h_1,h_2)\in H^0(\cO_{\FF_0}(3,4))\times H^0(\cO_{\PP^1}(1)) \times H^0(\cO_{\PP^1}(1)) \}\dashrightarrow \cM_8$$
sending $(g,h_1,h_2)$ to the isomorphism class of the double cover of $\FF_0$ branched along the divisor $f=0$ defined above. It follows that $\cM_8$ is unirational.

\subsection{$k=5$}
Let $X$ be the desingularization of a double cover of $\PP^2$ branched over a sextic $B$ with a simple node and admitting a tritangent line. Then $X$ is a K3 surface with
$$\begin{pmatrix}
2 &1 &0\\
1 &-2 &0\\
0 &0 &-2
\end{pmatrix} \cong U \oplus \langle -10\rangle \hookrightarrow \NS(X)$$
(see Lemma \ref{lemma:even}). Since $10$ is square-free, this embedding is primitive. Conversely, let $X$ be a K3 surface with $\NS(X)\cong U\oplus \langle -10\rangle$ and take a basis $\{H,L,C\}$ with intersection matrix as above. The linear system $|H|$ induces a morphism $X\rightarrow \PP^2$ of degree $2$ contracting $C$. Let $Y\rightarrow \PP^2$ denote the double cover obtained contracting $C$. Then $Y$ has a singular point of type $A_1$, so the branch locus $B\subseteq \PP^2$ has a node. Moreover $L$ is mapped onto a line of $\PP^2$ meeting $B$ with even multiplicities, so generically it will be a tritangent of $B$. Now, up to a projective transformation, we can assume that the tritangent line is given by $V(x_0)$, so that $B$ is given by an equation of the form
\begin{equation} \label{eq:5}
    f=x_0g(x_0,x_1,x_2)+h_1(x_1,x_2)^2h_2(x_1,x_2)^2h_3(x_1,x_2)^2.
\end{equation}
We can also assume that the node of $B$ is located at $P=(1:0:0)$. This forces the coefficients of $g$ of the terms $x_0^5,x_0^4x_1$ and $x_0^4x_2$ to be zero. We denote by $\mathcal{Q}_{10}$ the linear subspace of $H^0(\cO_{\PP^2}(5))$ consisting of all such polynomials. Then there exists a dominant rational map
$$\cP_{10}=\mathcal{Q}_{10}\times H^0(\cO_{\PP^1}(1))^3 \dashrightarrow \cM_{10}.$$
sending $(g,h_1,h_2,h_3)$ to the isomorphism class of the double cover of $\PP^2$ branched over $f$ defined as in equation (\ref{eq:5}). As $\cP_{10}$ is an affine space, $\cM_{10}$ is unirational.

\subsection{$k=6$}
By Proposition \ref{prop:F0} and Remark \ref{oss:all}, a general such K3 surface is the double cover of $\FF_0$ branched over a $(4,4)$-curve $B$ admitting a smooth $(1,1)$-curve $C$ intersecting $B$ in $4$ points with multiplicity $2$. We can choose coordinates on $\FF_0$ so that $C=V(x_0y_1-x_1y_0)$ and $B$ does not pass through the point $((0:1),(0:1))\in C$, so that the intersection $B\cap C$ is contained in the chart $W=\{x_0\ne 0, y_0\ne 0\}$, with coordinates $(1:u),(1:v)$. Say that $B$ is given by the equation
$$f(x_0,x_1,y_0,y_1)=\sum_{i+j=k+l=4}{\alpha_{ijkl}x_0^ix_1^jy_0^ky_1^l}$$
with $\alpha_{0404}=1$. Since $C|_W=V(u-v),$ the intersection $B\cap C\subseteq W$ is given by the vanishing of
$$g(u)=f(1,u,1,u)=\sum_{i+j=k+l=4}{\alpha_{ijkl}u^{j+l}}=\sum_{\eta=0}^8{\beta_\eta u^\eta},$$
where $\beta_\eta=\sum_{j+l=\eta}{\alpha_{ijkl}}$ and $\beta_8=\alpha_{0404}=1$. Now $g(u)$ has $4$ double roots at $u=\varepsilon_1,\varepsilon_2,\varepsilon_3,\varepsilon_4$ if and only if
$$g(u)=(u-\varepsilon_1)^2(u-\varepsilon_2)^2(u-\varepsilon_3)^2(u-\varepsilon_4)^2.$$
The choice of $\varepsilon_1,\varepsilon_2,\varepsilon_3,\varepsilon_4$ uniquely determines the coefficients $\beta_\eta$ for $\eta\le 7$, which in turn uniquely determine $8$ of the $\alpha_{ijkl}$. The $17$ other coefficients $\alpha_{ijkl}$ are free parameters so, if we denote them by $\alpha'_1,\ldots,\alpha_{17}'$, we have a rational dominant map
$$\cP_{12}:=\{(\varepsilon_i,\alpha_j')\in (\Aff^1)^4 \times (\Aff^1)^{17} \}\dashrightarrow \cM_{12}.$$
$\cP_{12}$ is an affine space, so $\cM_{12}$ is unirational.

\subsection{$k=8$}
By Proposition \ref{prop:F0} and Remark \ref{oss:all}, a general such K3 surface is the double cover of $\FF_0$ branched over a $(4,4)$-curve $B$ admitting a smooth $(1,2)$-curve $C$ intersecting $B$ in $6$ points with multiplicity $2$. We can choose coordinates on $\PP^3$ such that $\FF_0=V(x_0x_3-x_1x_2)$ and $C$ is the twisted cubic curve $V(x_0x_3-x_1x_2,x_1^2-x_0x_2,x_2^2-x_1x_3)$. We may also assume that $B$ does not pass through the point $(0:0:0:1)\in C$, so that the intersection $B\cap C$ is contained in the chart $W=\{x_0\ne 0\}\subseteq \PP^3$ with coordinates $(1:u:v:w)$. Then $C|_{W}=V(v-u^2,w-u^3)$, so if $B$ is given by a quartic $f(x_0,x_1,x_2,x_3)\in H^0(\cO_{\PP^3}(4))$, the intersection $B\cap C\subseteq W$ is given by the vanishing of
$$g(u)=f(1,u,u^2,u^3).$$
Now an argument as in the case $k=6$ shows that $\cM_{16}$ is unirational.\\

Now we will deal with the remaining values of $k$ by studying appropriate projective models of $U\oplus \langle -2k \rangle$-polarized K3 surfaces as complete intersections. The strategy will be similar for every degree. We are going to need the following lemma.

\begin{lemma} \label{lemma:fine}
The parameter space $\mathcal{C}_{d, n}$ of rational normal curves in $\PP^d$ of degree $d$ passing through $0\leq n\leq 3$ points of $\PP^d$ in general position is unirational and non-empty.
\end{lemma}
\begin{proof}
Rational normal curves of degree $d$ are parametrized by $d+1$ linearly independent forms $A_j\in H^0(\cO_{\PP^1}(d))$, inducing the embedding
$$\begin{array}{cccc}
\varphi_{\bm{A}}: &\PP^1 &\hookrightarrow &\PP^d\\
&(u:v) &\mapsto &(A_0(u,v):\ldots:A_d(u,v)).
\end{array}$$
This shows that the parameter space $\mathcal{C}_{d,0}$ of rational normal curves is unirational, as there exists a dominant rational map
$$ H^0(\cO_{\PP^1}(d))^{d+1}\dashrightarrow \mathcal{C}_{d,0}.$$

Let $P_1,\ldots,P_n\in \PP^d$ be $n$ points in general position. The curve $C=\varphi_{\bm{A}}(\PP^1)$ passes through $P_1,\ldots,P_n$ if and only if there exist $R_1,\ldots,R_n\in \PP^1$ mapped to $P_1,\ldots,P_n$ under $\varphi_{\bm{A}}$. Since $n \le 3$, we can suppose that $\{R_1,\ldots,R_n\}$ is a subset of $\{(1:0),(0:1),(1:1)\}$ up to automorphism of $\PP^1$.

Thus $C$ passes through $P_1,\ldots,P_n$ if and only if
$$P_i=(A_0(R_i):\ldots:A_d(R_i))$$
for $1\le i \le n$. These equations define a linear subspace $\widetilde{\mathcal{C}}_{d,n}$ of $H^0(\cO_{\PP^1}(d))^{d+1}$. There exists a dominant rational map
$$\widetilde{\mathcal{C}}_{d,n} \dashrightarrow \mathcal{C}_{d,n},$$
hence the space $\mathcal{C}_{d,n}$ is unirational.
In order to see that $\mathcal{C}_{d,n}\ne \varnothing$, recall that there always exists a rational normal curve of degree $d$ in $\PP^d$ passing through $d+3$ points.
\end{proof}

\subsection{Quartic surfaces in $\PP^3$: $k\in \{10,13,16,19,26\}$}
We start with the case of polarizations of degree $4$. We consider the following isomorphisms of lattices:
$$ \begin{pmatrix}
4 &1 &1\\
1 &-2 &0\\
1 &0 &-2
\end{pmatrix} \cong U \oplus \langle -20\rangle, \quad \begin{pmatrix}
4 &2 &1\\
2 &-2 &0\\
1 &0 &-2
\end{pmatrix} \cong U \oplus \langle -26\rangle, \quad
\begin{pmatrix}
4 &3 &1\\
3 &-2 &2\\
1 &2 &-2
\end{pmatrix} \cong U \oplus \langle -32\rangle, $$
$$\begin{pmatrix}
4 &3 &1\\
3 &-2 &1\\
1 &1 &-2
\end{pmatrix} \cong U \oplus \langle -38\rangle, \quad
\begin{pmatrix}
4 &3 &3\\
3 &-2 &0\\
3 &0 &-2
\end{pmatrix} \cong U \oplus \langle -52\rangle.$$

Let $X\subseteq \PP^3$ be a smooth quartic surface containing two rational normal curves $C_1,C_2$ of degree $d_1,d_2\in \{1,2,3\}$ respectively, intersecting at $n\in \{0,1,2\}$ points. The numbers $d_1,d_2,n$ depend on $k$, as specified by the lattices above. Denote by $H$ the hyperplane class. Suppose that the lattice $\langle H,C_1,C_2\rangle$ has as intersection form one of the lattices above, and embeds into $\NS(X)$. 

We claim that this embedding is always primitive. If $k\in \{10,13,19,26\}$, then the lattice $U\oplus \langle -2k \rangle$ admits no non-trivial overlattices by \cite[Proposition 1.4.1]{Nik79}, so the embedding is necessarily primitive. If instead $k=16$, it is easy to see that the embedding is not primitive if and only if the class $C_1+C_2$ is divisible in $\NS(X)$. Since $C_1+C_2$ has square $0$ and is reduced and connected on $X$, it is primitive in $\NS(X)$, and hence the embedding is also primitive. This shows that $X$ is a $U\oplus \langle -2k \rangle$-polarized K3 surface.

Conversely, let $X$ be a K3 surface with $\NS(X)=\langle H,C_1,C_2\rangle$ and intersection matrix
$$\begin{pmatrix}
4 &d_1 &d_2\\
d_1 &-2 &n\\
d_2 &n &-2
\end{pmatrix}$$
belonging to the aforementioned list. The divisor $H$ is ample, and actually very ample by \cite[Theorem 5.2]{Sai74}, since there are no divisors $D\in \NS(X)$ of square $0$ and $DH=2$. Therefore the linear system $|H|$ induces an embedding $X\hookrightarrow \PP^3$ realizing $X$ as a quartic surface containing two rational normal curves of degree $d_1,d_2$ respectively, intersecting at $n$ points. We denote by $\cP_{2k}\subseteq |\cO_{\PP^3}(4)|$ the parameter space of such quartic surfaces in $\PP^3$. This shows that there exists a rational map $\cP_{2k}\dashrightarrow \cM_{2k}$, which is dominant as soon as $\cP_{2k}$ contains a smooth surface. In this case, the expected dimension is $\dim(\cP_{2k})=\dim(\cM_{2k})+\dim(\mathrm{SL}(4, \CC))=32$.

In order to prove that $\cP_{2k}$ contains a smooth surface, we show that there actually exist smooth quartic surfaces $X\subseteq \PP^3$ containing the union $C_1 \cup C_2$. This can be done by a concrete argument. For instance, consider the case $k=19$, where $d_1=3$, $d_2=1$ and $n=1$. Up to an automorphism of $\PP^3$, we can assume that $C_1=\{(u^3:u^2v:uv^2:v^3) \mid (u:v)\in \PP^1\}$ and that $C_1$ meets the line $C_2$ at $P=(1:0:0:0)$. Therefore we can write $C_2=V(l_1,l_2)$ for some linear equations $l_1,l_2 \in H^0(\cO_{\PP^3}(1))$ vanishing at $P$. Now a quartic surface $X=V(f)\subseteq \PP^3$ contains the union $C_1\cup C_2$ if and only if $f=l_1 k_1 +l_2 k_2$ for some $k_1,k_2 \in H^0(\cO_{\PP^3}(3))$ and
$$f(u^3,u^2v,uv^2,v^3)\equiv 0$$
as a polynomial in $(u:v)$. A straightforward computation shows that there exist such polynomials $f$ defining smooth quartic surfaces.

Alternatively one can show that $\dim(\cP_{2k})=32$ by a dimension count. This implies that $\cP_{2k}$ contains a smooth surface: otherwise, the image of the map $\cP_{2k} \dashrightarrow \cM_{2k}$ would have dimension $<17$, thus $\mathrm{dim}(\cP_{2k})< 17+\mathrm{dim}(\mathrm{SL}(4, \CC))<32$, a contradiction.

Consider for instance the case $k=26$, where $d_1=d_2=3$ and $n=0$. The short exact sequence
$$0 \rightarrow \cI_{C_1 \cup C_2}(4) \rightarrow \cO_{\PP^3}(4) \rightarrow \cO_{C_1 \cup C_2}(4) \rightarrow 0$$
induces the exact sequence
$$0 \rightarrow H^0(\cI_{C_1 \cup C_2}(4)) \rightarrow H^0(\cO_{\PP^3}(4)) \rightarrow H^0(\cO_{C_1 \cup C_2}(4)) \rightarrow H^1(\cI_{C_1 \cup C_2}(4)).$$
The last term $H^1(\cI_{C_1 \cup C_2}(4))=0$, since $C_1,C_2$ are both projectively normal and disjoint. Therefore
$$h^0(\cI_{C_1 \cup C_2}(4))=h^0(\cO_{\PP^3}(4))-h^0(\cO_{C_1 \cup C_2}(4))=35-26=9,$$
so
$$\dim(\cP_{52})=2\dim(\mathcal{T})+(h^0(\cI_{C_1 \cup C_2}(4))-1)=2\cdot 12 +8=32,$$
where $\mathcal{T}=\mathrm{SL}(4, \CC)/\mathrm{SL}(2, \CC)$ denotes the moduli space of twisted cubics in $\PP^3$.

\begin{rmk}
For the dimension count, it is important to notice that the general element in $\cP_{2k}$ is irreducible. Indeed, the space of reducible quartic surfaces in $\PP^3$ has dimension $22 < 32$.
\end{rmk}

We now study the space $\cP_{2k}$ and prove its unirationality, implying in particular that $\cM_{2k}$ is unirational, too. Fix numbers $d_1,d_2,n$ as above. Up to automorphism of $\PP^3$ we can fix a rational normal curve $C_1$ of degree $d_1$. First we choose a set of points $P_1,\ldots,P_n\in C_1$, which will be the points of intersection of $C_1$ and $C_2$. Then we choose a linear subspace $\Pi_2\subseteq \PP^3$ of dimension $d_2 \ge n-1$ containing $P_1,\ldots,P_n$, and a rational normal curve $C_2 \subseteq \Pi_2$ of degree $d_2$, passing through $P_1,\ldots,P_n$. 

Using the notations of Lemma \ref{lemma:fine}, we consider the following incidence varieties:

$$\cP'_{2k}:=\{(P_1,\ldots,P_n,\Pi_2,\bm{A},X) \in \mathrm{Sym}^n(C_1)\times \mathrm{Gr}(d_2,3)\times \widetilde{\mathcal{C}}_{d_2,n} \times |\cO_{\PP^3}(4)|:$$
$$ P_1,\ldots,P_n \in  C_2=\varphi_{\bm{A}}(\PP^1)\subseteq \Pi_2, \  X \supseteq C_1 \cup C_2\},$$
$$\mathcal{Z}_1:=\{(P_1,\ldots,P_n,\Pi_2,\bm{A}) \in \mathrm{Sym}^n(C_1)\times \mathrm{Gr}(d_2,3)\times \widetilde{\mathcal{C}}_{d_2,n} : \Pi_2\ni P_1,\ldots,P_n\},$$
$$\mathcal{Z}_2:=\{(P_1,\ldots,P_n,\Pi_2) \in \mathrm{Sym}^n(C_1)\times \mathrm{Gr}(d_2,3): \Pi_2\ni P_1,\ldots,P_n\}.$$

They are related as shown in the diagram below:
$$\begin{tikzcd}
\cP'_{2k} \arrow[dashed]{d}{p_1} \arrow{r}{\psi_1} &\mathcal{Z}_1 \arrow{r}{\psi_2} &\mathcal{Z}_2 \arrow{r}{\psi_3} & \mathrm{Sym}^n(C_1)\\
\cP_{2k} \arrow[dashed]{d}{p_2}\\
\cM_{2k}
\end{tikzcd}$$
The three maps $\psi_1,\psi_2$ and $\psi_3$ are the obvious forgetful maps. The map $p_1$ is the projection onto the last factor, while $p_2$ sends a smooth quartic surface to its isomorphism class.

In order to prove the unirationality of $\cM_{2k}$ we show that $\cP'_{2k}$ is rational. The variety $\mathrm{Sym}^n(C_1)\cong \PP^n$ is rational and $\mathcal{Z}_2$ is also rational, since it is a $\mathrm{Gr}(d_2-n, 3-n)$-bundle over $\mathrm{Sym}^n(C_1)$. Then $\mathcal{Z}_1$ is a projective bundle over $\mathcal{Z}_2$ with fiber isomorphic to $\widetilde{\mathcal{C}}_{d_2,n}$ in the proof of Lemma \ref{lemma:fine}. Finally $\mathcal{P}'_{2k}$ is a projective bundle over $\mathcal{Z}_1$ with fiber $\PP H^0(\cI_{C_1 \cup C_2}(4))$. In conclusion, it follows that $\cP'_{2k}$ is rational. Therefore $\cM_{2k}$ is unirational, since $p_1$ and $p_2$ are dominant rational maps.

\subsection{Nodal quartic surfaces in $\PP^3$: $k\in \{7, 17\}$}
We consider the following isomorphisms of lattices:
$$\begin{pmatrix}
4 &1 &0\\
1 &-2 &1\\
0 &1 &-2
\end{pmatrix} \cong U \oplus \langle -14\rangle, \quad  \begin{pmatrix}
4 &3 &0\\
3 &-2 &0\\
0 &0 &-2
\end{pmatrix} \cong U \oplus \langle -34\rangle.$$

Let $X'\subseteq \PP^3$ be a quartic surface containing a line (resp. a twisted cubic) $C_1$ and with an $A_1$ singularity on the line (resp. outside the twisted cubic). The desingularization $X$ of $X'$ contains the exceptional divisor $C_2\cong \PP^1$ over the $A_1$ singularity of $X'$. If $H$ denotes the hyperplane class, the lattice $\langle H,C_1,C_2 \rangle$ is isometric to one of the two lattices above and embeds into $\NS(X)$. The embedding is primitive, since $14$ and $34$ are both square-free. Therefore $X$ is a $U\oplus \langle -2k \rangle$-polarized K3 surface.

Conversely, let $X$ be a K3 surface with $\NS(X)=\langle H,C_1,C_2\rangle$ and intersection matrix
$$\begin{pmatrix}
4 &d &0\\
d &-2 &e\\
0 &e &-2
\end{pmatrix}$$
belonging to the list above. By \cite[Theorem 5.2]{Sai74} the linear system $|H|$ induces a rational map $X \dashrightarrow \PP^3$ contracting $C_2$ to a point $P$, and embedding $X\smallsetminus C_2 \hookrightarrow \PP^3$. This realizes $X$ as a nodal quartic surface containing a rational normal curve $C_1$ of degree $d$ and an $A_1$ singularity on (resp. outside) $C_1$ if $e=1$ (resp. $e=0$). We denote by $\cP_{2k}\subseteq |\cO_{\PP^3}(4)|$ the space of such nodal quartic surfaces in $\PP^3$. This shows that there exists a rational map $\cP_{2k}\dashrightarrow \cM_{2k}$, which is dominant as soon as $\cP_{2k}$ contains an irreducible surface with a unique $A_1$-singularity. In order to prove that, we can argue as in the case of smooth quartic surfaces above. In this case, the expected dimension is $\dim(\cP_{2k})=\dim(\cM_{2k})+\dim(\mathrm{SL}(4, \CC))=32$.

We now study the space $\cP_{2k}$ and prove its unirationality, implying in particular that $\cM_{2k}$ is unirational, too. Up to automorphism of $\PP^3$ we fix a rational normal curve $C_1$ of degree $d$, and consider the following incidence varieties:

$$\cP'_{14}:=\{(P,X) \in C_1\times |\cO_{\PP^3}(4)|:  X \supseteq C_1, \  P\in X \text{ is an $A_1$ singularity} \},$$
$$\cP'_{34}:=\{(P,X) \in \PP^3\times |\cO_{\PP^3}(4)|:  X \supseteq C_1, \ P\notin C_1, \  P\in X \text{ is an $A_1$ singularity} \}.$$
The varieties are open dense subsets of projective bundles over $C_1\cong \PP^1$ (resp. $\PP^3$), since imposing a singular point at $P$ for the quartic surface $X=V(f)$ is a linear condition on the coefficients of $f$, and a singular point is generically of type $A_1$. Therefore $\cP'_{14}$ and $\cP'_{34}$ are both rational. There exist dominant rational maps
$$\cP_{2k}' \stackrel{p_1}{\dashrightarrow} \cP_{2k} \stackrel{p_2}{\dashrightarrow} \cM_{2k},$$
where $p_1$ is the projection onto the last factor, and $p_2$ sends a nodal quartic surface to the isomorphism class of its desingularization. We conclude that $\cM_{2k}$ is unirational.

\subsection{Complete intersections of a quadric and a cubic in $\PP^4$: $k\in \{9, 21, 25, 29, 37\}$}
We continue with the case of polarizations of degree $6$. We consider the following isomorphisms of lattices:
$$ \begin{pmatrix}
6 &2 &1\\
2 &-2 &2\\
1 &2 &-2
\end{pmatrix} \cong U \oplus \langle -18\rangle, \quad \begin{pmatrix}
6 &2 &2\\
2 &-2 &1\\
2 &1 &-2
\end{pmatrix} \cong U \oplus \langle -42\rangle, \quad
\begin{pmatrix}
6 &3 &2\\
3 &-2 &2\\
2 &2 &-2
\end{pmatrix} \cong U \oplus \langle -50\rangle, $$
$$\begin{pmatrix}
6 &4 &1\\
4 &-2 &0\\
1 &0 &-2
\end{pmatrix} \cong U \oplus \langle -58\rangle, \quad
\begin{pmatrix}
6 &4 &2\\
4 &-2 &1\\
2 &1 &-2
\end{pmatrix} \cong U \oplus \langle -74\rangle.$$

Let $X\subseteq \PP^4$ be a smooth complete intersection of a quadric and a cubic containing two rational normal curves $C_1,C_2$ of degree $d_1,d_2\in \{1,2,3,4\}$ respectively, intersecting at $n\in \{0,1,2\}$ points. The numbers $d_1,d_2,n$ depend on $k$, as specified by the lattices above. Denote by $H$ the hyperplane class. Then the lattice $\langle H,C_1,C_2\rangle$ has intersection form of one of the types above, and embeds into $\NS(X)$. 

We claim that this embedding is always primitive. If $k\in \{21,29,37\}$, then the lattice $U\oplus \langle -2k \rangle$ admits no non-trivial overlattices by \cite[Proposition 1.4.1]{Nik79}, so the embedding is necessarily primitive. If instead $k\in \{9,25\}$, it is easy to see that the embedding is not primitive if and only if the class $C_1+C_2$ is divisible in $\NS(X)$. Since $C_1+C_2$ has square $0$ and is reduced and connected on $X$, it is primitive in $\NS(X)$, and hence the embedding is also primitive. This shows that $X$ is a $U\oplus \langle -2k \rangle$-polarized K3 surface.

Conversely, let $X$ be a K3 surface with $\NS(X)=\langle H,C_1,C_2\rangle$ and intersection matrix
$$\begin{pmatrix}
6 &d_1 &d_2\\
d_1 &-2 &n\\
d_2 &n &-2
\end{pmatrix}$$
belonging to the list above. The divisor $H$ is ample, and actually very ample by \cite[Theorem 5.2]{Sai74}, since there are no divisors $D\in \NS(X)$ of square $0$ and $DH=2$. Therefore the linear system $|H|$ induces an embedding $X\hookrightarrow \PP^4$ realizing $X$ as a smooth complete intersection of a quadric and a cubic containing two rational normal curves of degree $d_1,d_2$ respectively, intersecting at $n$ points. We denote by $\cP_{2k}$ the space of such complete intersections in $\PP^4$. The discussion above shows that there exists a rational map
$$\cP_{2k}\dashrightarrow \cM_{2k},$$
which is dominant as soon as $\cP_{2k}$ contains a smooth surface.
In this case, the expected dimension is $\dim(\cP_{2k})=\dim(\cM_{2k})+\dim(\mathrm{SL}(5, \CC))=41$.

In order to prove that $\cP_{2k}$ contains a smooth element, we can argue as in the case of quartic surfaces, using concrete computations or a dimension count. Let us go through one example, and we leave the rest to the reader. Consider $k=9$, so that $d_1=2$, $d_2=1$ and $n=2$. The union $C_1\cup C_2$ is contained in a plane $\Pi$, so we have a short exact sequence
$$0 \rightarrow \cI_\Pi(m) \rightarrow \cI_{C_1 \cup C_2}(m) \rightarrow \cO_\Pi(m-3)\rightarrow 0$$
for every $m\in \ZZ$. Since the $h^1$ of the first and the last term are zero for all $m\in \ZZ$, we have that $h^1(\cI_{C_1 \cup C_2}(m))=0$ for all $m \in \ZZ$. Now we can proceed with the dimension count. Choosing the conic $C_1$ counts for $11$ parameters ($6$ for the plane $\Pi$, $5$ for the conic); the choice of the line $C_2$ depends uniquely on the choice of the two points of intersection with $C_1$. The choice of a quadric containing $C_1 \cup C_2$ depends on
$$h^0(\cI_{C_1 \cup C_2}(2))-1=h^0(\cO_{\PP^4}(2))-h^0(\cO_{C_1 \cup C_2}(2))-1= 8$$
parameters, while the choice of a cubic depends on
$$h^0(\cI_{C_1 \cup C_2}(3))-6=h^0(\cO_{\PP^4}(3))-h^0(\cO_{C_1 \cup C_2}(3))-6= 20$$
parameters, since we want the cubic not to be a multiple of the quadric. Summing up,
$$\dim(\cP_{18})=11+2+8+20=41,$$
as expected.

\begin{rmk} 
For the dimension count, it is important to notice that the general element in $\cP_{2k}$ is irreducible. This is easy if $k \ne 9$: the general quadric in $|\mathcal{I}_{C_1 \cup C_2}(2)|$ and the general cubic in $|\mathcal{I}_{C_1 \cup C_2}(3)|$ are smooth and do not contain a plane, so their intersection is irreducible. If $k=9$, all quadrics in $|\mathcal{I}_{C_1 \cup C_2}(2)|$ contain the plane $\Pi$ spanned by $C_1$ and $C_2$. However, the general such quadric is irreducible, so it suffices to show that the general cubic in $|\mathcal{I}_{C_1 \cup C_2}(3)|$ does not contain a plane. Clearly the general such cubic does not contain any plane other than $\Pi$. Moreover a straightforward computation shows that $|\mathcal{I}_{\Pi}(3)|$ has codimension $1$ in $|\mathcal{I}_{C_1 \cup C_2}(3)|$.
\end{rmk}

We now study the space $\cP_{2k}$ and prove its unirationality, implying in particular that $\cM_{2k}$ is unirational, too. Fix numbers $d_1,d_2,n$ as above. Up to automorphism of $\PP^4$ we can fix a rational normal curve $C_1$ of degree $d_1$. First we choose a set of points $P_1,\ldots,P_n\in C_1$, which will be the points of intersection of $C_1$ and $C_2$. Then we choose a linear subspace $\Pi_2\subseteq \PP^4$ of dimension $d_2 \ge n-1$ containing $P_1,\ldots,P_n$, and a rational normal curve $C_2 \subseteq \Pi_2$ of degree $d_2$, passing through $P_1,\ldots,P_n$. 

Using the notations of Lemma \ref{lemma:fine}, we consider the following incidence varieties:

$$\cP'_{2k}:=\{(P_1,\ldots,P_n,\Pi_2,\bm{A},Q, K) \in \mathrm{Sym}^n(C_1)\times \mathrm{Gr}(d_2,4)\times \widetilde{\mathcal{C}}_{d_2,n} \times |\cO_{\PP^4}(2)|\times |\cO_{\PP^4}(3)|:$$
$$ P_1,\ldots,P_n \in  C_2=\varphi_{\bm{A}}(\PP^1)\subseteq \Pi_2, \  Q \cap K \supseteq C_1 \cup C_2\},$$
$$\mathcal{Z}_1:=\{(P_1,\ldots,P_n,\Pi_2,\bm{A}) \in \mathrm{Sym}^n(C_1)\times \mathrm{Gr}(d_2,4)\times \widetilde{\mathcal{C}}_{d_2,n} : \Pi_2\ni P_1,\ldots,P_n\},$$
$$\mathcal{Z}_2:=\{(P_1,\ldots,P_n,\Pi_2) \in \mathrm{Sym}^n(C_1)\times \mathrm{Gr}(d_2,4): \Pi_2\ni P_1,\ldots,P_n\}.$$

They are related as shown in the diagram below:
$$\begin{tikzcd}
\cP'_{2k} \arrow[dashed]{d}{p_1} \arrow{r}{\psi_1} &\mathcal{Z}_1 \arrow{r}{\psi_2} &\mathcal{Z}_2 \arrow{r}{\psi_3} & \mathrm{Sym}^n(C_1)\\
\cP_{2k} \arrow[dashed]{d}{p_2}\\
\cM_{2k}
\end{tikzcd}$$
The three maps $\psi_1,\psi_2$ and $\psi_3$ are the obvious forgetful maps. The map $p_1$ sends $(P_1,\ldots,P_n$, $\Pi_2,\bm{A},Q,K)$ to $Q \cap K$, while $p_2$ sends a smooth complete intersection of a quadric and a cubic in $\PP^4$ to its isomorphism class.

In order to prove the unirationality of $\cM_{2k}$ we show that $\cP'_{2k}$ is rational. The variety $\mathrm{Sym}^n(C_1)\cong \PP^n$ is rational and $\mathcal{Z}_2$ is also rational, since it is a $\mathrm{Gr}(d_2-n, 4-n)$-bundle over $\mathrm{Sym}^n(C_1)$. Then $\mathcal{Z}_1$ is a projective bundle over $\mathcal{Z}_2$ with fiber isomorphic to $\widetilde{\mathcal{C}}_{d_2,n}$ in the proof of Lemma \ref{lemma:fine}. Finally $\mathcal{P}'_{2k}$ is a projective bundle over $\mathcal{Z}_1$ with fiber $\PP H^0(\cI_{C_1 \cup C_2}(2)) \times \PP H^0(\cI_{C_1 \cup C_2}(3))$. In conclusion, it follows that $\cP'_{2k}$ is rational. Therefore $\cM_{2k}$ is unirational, since $p_1$ and $p_2$ are dominant rational maps.

\subsection{Complete intersections of three quadrics in $\PP^5$: $k\in \{31,$ $34,$ $36,$ $39,$ $41,$ $43,$ $49,$ $59,$ $61,$ $64\}$}
We conclude with the case of polarizations of degree $8$. We consider the following isomorphisms of lattices:
$$\begin{pmatrix}
8 &3 &2\\
3 &-2 &1\\
2 &1 &-2
\end{pmatrix} \cong U \oplus \langle -62\rangle, \quad \begin{pmatrix}
8 &3 &3\\
3 &-2 &0\\
3 &0 &-2
\end{pmatrix} \cong U \oplus \langle -68\rangle, \quad \begin{pmatrix}
8 &3 &3\\
3 &-2 &2\\
3 &2 &-2
\end{pmatrix} \cong U \oplus \langle -72\rangle,$$
$$\begin{pmatrix}
8 &3 &3\\
3 &-2 &1\\
3 &1 &-2
\end{pmatrix} \cong U \oplus \langle -78\rangle, \quad \begin{pmatrix}
8 &4 &3\\
4 &-2 &3\\
3 &3 &-2
\end{pmatrix} \cong U \oplus \langle -82\rangle, \quad \begin{pmatrix}
8 &5 &1\\
5 &-2 &1\\
1 &1 &-2
\end{pmatrix} \cong U \oplus \langle -86\rangle,$$ 
$$\begin{pmatrix}
8 &4 &3\\
4 &-2 &2\\
3 &2 &-2
\end{pmatrix} \cong U \oplus \langle -98\rangle, \quad
\begin{pmatrix}
8 &5 &3\\
5 &-2 &3\\
3 &3 &-2
\end{pmatrix} \cong U \oplus \langle -118\rangle, \quad 
\begin{pmatrix}
8 &5 &3\\
5 &-2 &1\\
3 &1 &-2
\end{pmatrix} \cong U \oplus \langle -122\rangle,$$
$$\begin{pmatrix}
8 &5 &3\\
5 &-2 &2\\
3 &2 &-2
\end{pmatrix} \cong U \oplus \langle -128\rangle,$$

Let $X\subseteq \PP^5$ be a smooth complete intersection of three quadrics containg two rational normal curves $C_1,C_2$ of degree $d_1,d_2 \in \{1,2,3,4,5\}$ respectively, intersecting at $n \in \{0,1,2,3\}$ points. The numbers $d_1,d_2,n$ depend on $k$, as specified by the lattices above. Denote by $H$ the hyperplane class. Then the lattice $\langle H,C_1,C_2\rangle$ has intersection form of one of the types above, and embeds into $\NS(X)$. We claim that this embedding is always primitive. If $k\in \{31,34,39,41,43,59,61\}$, then the lattice $U\oplus \langle -2k \rangle$ admits no non-trivial overlattices by \cite[Proposition 1.4.1]{Nik79}, so the embedding is necessarily primitive. If instead $k\in \{36,49,64\}$, it is easy to see that the embedding is not primitive if and only if the class $C_1+C_2$ is divisible in $\NS(X)$. Since $C_1+C_2$ has square $0$ and is reduced and connected on $X$, it is primitive in $\NS(X)$, and hence the embedding is also primitive. This shows that $X$ is a $U\oplus \langle -2k \rangle$-polarized K3 surface.

Conversely, let $X$ be a K3 surface with $\NS(X)=\langle H,C_1,C_2\rangle$ and intersection matrix
$$\begin{pmatrix}
8 &d_1 &d_2\\
d_1 &-2 &n\\
d_2 &n &-2
\end{pmatrix}$$
belonging to the list above. The divisor $H$ is ample, and actually very ample by \cite[Theorem 5.2]{Sai74}, since there are no divisors $D\in \NS(X)$ of square $0$ and $DH=2$. Therefore the linear system $|H|$ induces an embedding $X\hookrightarrow \PP^5$ realizing $X$ as a smooth complete intersection of three quadrics containing two rational normal curves of degree $d_1,d_2$ respectively, intersecting at $n$ points. We denote by $\cP_{2k}$ the space of such complete intersections in $\PP^5$. The previous discussion shows that there exists a rational map $\cP_{2k}\dashrightarrow \cM_{2k}$, which is dominant as soon as $\cP_{2k}$ contains a smooth surface. In this case, the expected dimension $\dim(\cP_{2k})=\dim(\cM_{2k})+\dim(\mathrm{SL}(6, \CC))=52$.

In order to prove that $\cP_{2k}$ contains a smooth element, we can argue as in the case of polarizations of degree $4$ or $6$, using concrete computations or a dimension count. We leave the details to the reader.

\begin{rmk}
For the dimension count, it is important to notice that the general element in $\cP_{2k}$ is an irreducible surface. We can argue as follows. First we notice that $h^0(I_{C_1 \cup C_2}(2))\ge 3$ in all cases. Since the general quadric in $|\mathcal{I}_{C_1 \cup C_2}(2)|$ is irreducible and does not contain a $\PP^3$, the intersection of two such general quadrics is an irreducible threefold. Next, we have to show that the intersection of three general quadrics in $|\mathcal{I}_{C_1 \cup C_2}(2)|$ does not contain a plane, a quadric surface, or a cubic surface. This is clear if this surface $S$ contains neither $C_1$ nor $C_2$. If instead $S$ is a plane (or a quadric/cubic surface) containing $C_1$, a straightforward dimension count shows that the general quadric in $|\mathcal{I}_{C_1 \cup C_2}(2)|$ does not contain $S$.
\end{rmk}

We now study the space $\cP_{2k}$ and prove its unirationality, implying in particular that $\cM_{2k}$ is unirational, too. Fix numbers $d_1,d_2,n$ as above. Up to automorphism of $\PP^5$ we can fix a rational normal curve $C_1$ of degree $d_1$. First we choose a set of points $P_1,\ldots,P_n\in C_1$, which will be the points of intersection of $C_1$ and $C_2$. Then we choose a linear subspace $\Pi_2\subseteq \PP^5$ of dimension $d_2 \ge n-1$ containing $P_1,\ldots,P_n$, and a rational normal curve $C_2 \subseteq \Pi_2$ of degree $d_2$, passing through $P_1,\ldots,P_n$. 

Using the notations of Lemma \ref{lemma:fine}, we consider the following incidence varieties:

$$\cP'_{2k}:=\{(P_1,\ldots,P_n,\Pi_2,\bm{A},Q_1,Q_2,Q_3) \in \mathrm{Sym}^n(C_1)\times \mathrm{Gr}(d_2,5)\times \widetilde{\mathcal{C}}_{d_2,n} \times |\cO_{\PP^5}(2)|^3:$$
$$ P_1,\ldots,P_n \in  C_2=\varphi_{\bm{A}}(\PP^1)\subseteq \Pi_2, \  Q_1\cap Q_2 \cap Q_3 \supseteq C_1 \cup C_2\},$$
$$\mathcal{Z}_1:=\{(P_1,\ldots,P_n,\Pi_2,\bm{A}) \in \mathrm{Sym}^n(C_1)\times \mathrm{Gr}(d_2,5)\times \widetilde{\mathcal{C}}_{d_2,n} : \Pi_2\ni P_1,\ldots,P_n\},$$
$$\mathcal{Z}_2:=\{(P_1,\ldots,P_n,\Pi_2) \in \mathrm{Sym}^n(C_1)\times \mathrm{Gr}(d_2,5): \Pi_2\ni P_1,\ldots,P_n\}.$$

They are related as shown in the diagram below:
$$\begin{tikzcd}
\cP'_{2k} \arrow[dashed]{d}{p_1} \arrow{r}{\psi_1} &\mathcal{Z}_1 \arrow{r}{\psi_2} &\mathcal{Z}_2 \arrow{r}{\psi_3} & \mathrm{Sym}^n(C_1)\\
\cP_{2k} \arrow[dashed]{d}{p_2}\\
\cM_{2k}
\end{tikzcd}$$
The three maps $\psi_1,\psi_2$ and $\psi_3$ are the obvious forgetful maps. The map $p_1$ sends $(P_1,\ldots,P_n$, $\Pi_2,\bm{A},Q_1,Q_2,Q_3)$ to $Q_1 \cap Q_2 \cap Q_3$, while $p_2$ sends a smooth complete intersection of three quadrics in $\PP^5$ to its isomorphism class.

In order to prove the unirationality of $\cM_{2k}$ we show that $\cP'_{2k}$ is rational. The variety $\mathrm{Sym}^n(C_1)\cong \PP^n$ is rational and $\mathcal{Z}_2$ is also rational, since it is a $\mathrm{Gr}(d_2-n, 5-n)$-bundle over $\mathrm{Sym}^n(C_1)$. Then $\mathcal{Z}_1$ is a projective bundle over $\mathcal{Z}_2$ with fiber isomorphic to $\widetilde{\mathcal{C}}_{d_2,n}$ in the proof of Lemma \ref{lemma:fine}. Finally $\mathcal{P}'_{2k}$ is a projective bundle over $\mathcal{Z}_1$ with fiber $\PP H^0(\cI_{C_1 \cup C_2}(2))^3$. In conclusion, it follows that $\cP'_{2k}$ is rational. Therefore $\cM_{2k}$ is unirational, since $p_1$ and $p_2$ are dominant rational maps.

\bibliographystyle{alpha}
\bibliography{References}

\end{document}